\documentclass[10pt,psamsfonts]{amsart}
\usepackage[margin=0.80in,bottom=0.90in,top=0.85in]{geometry}
\usepackage{eucal}
\usepackage{amsmath,amssymb,amsfonts,amsthm,mathrsfs}
\usepackage{graphicx, array, xcolor}
\usepackage{float}
\usepackage{xcolor}
\usepackage[foot]{amsaddr}
\usepackage{enumitem}
\usepackage{multicol}
\definecolor{mytealblue}{RGB}{0, 128, 128}
\colorlet{BLUE}{blue}

\usepackage[
    colorlinks=true,
    citecolor=mytealblue,
    linkcolor=mytealblue,
    urlcolor=mytealblue
]{hyperref}

\newcommand{\N}{[N]}

\newcommand{\R}{\mathbb{R}}
\newcommand{\I}{[N]}

\usepackage[mathlines]{lineno}

\newtheorem{theorem}{Theorem}[section]
\newtheorem{lemma}{Lemma}[section]
\theoremstyle{corollary}

\newtheorem{proposition}{Proposition}[section]
\theoremstyle{definition}

\newtheorem{remark}{Remark}[section]
\numberwithin{equation}{section}

\def\longdelete#1{}

\subjclass[2020]{35K57, 35C07, 35B40}

\begin{document}

\title[Propagation windows in plant-consumer systems]
{Propagation-Window Bounds in Degenerate Plant-Consumer Reaction-Diffusion Systems}

\author[Chen, C-C]{Chiun-Chuan Chen$^{1,2}$}
\address{$^1$ Department of Mathematics, National Taiwan University, Taiwan.}
\address{$^2$ National Center for Theoretical Science, Taipei, Taiwan.}
\email[Chen, C-C]{$^{1,2}$chchchen@math.ntu.edu.tw}

\author[Hsiao, T.-Y.]{Ting-Yang Hsiao$^3$}
\address{$^3$ International School for Advanced Studies (SISSA), Trieste, Italy (Corresponding author).}
\email[Hsiao, T.-Y.]{$^3$thsiao@sissa.it}

\author[Hung, L.-C.]{Li-Chang Hung$^4$}
\address{$^4$ Department of Mathematics, Soochow University, Taipei, Taiwan.}
\email[Hung, L.-C.]{$^4$lichang.hung@gmail.com}

\author[Li, H]{Haoyuan Li$^5$}
\address{$^5$ Department of Mathematics, University of Illinois Urbana--Champaign, USA.}
\email[Li, H]{$^5$hl56@illinois.edu}

\author[Wang, S-C]{Shun-Chieh Wang$^6$}
\address{$^6$ National Center for Theoretical Science, Taipei, Taiwan.}
\email[Wang, S-C]{$^6$rjaywang1130@ncts.ntu.edu.tw}

\keywords{Traveling waves;
critical wave speed;
reaction--diffusion systems;
plant--consumer models;
degenerate consumer dynamics;
upper-speed obstruction.}

\date{\today}

\begin{abstract}
We study traveling waves in high-dimensional plant--consumer reaction--diffusion systems with $N$ competing plants and $N$ associated consumers. We focus on the degenerate case in which consumer growth vanishes when plant populations are zero, so the standard Fisher--KPP linearization does not determine the leading-edge behavior. Under a weak-interaction condition, we identify a plant-driven lower threshold $s_{\mathrm P}$, a constructive upper threshold $s_{\mathrm C}^{\mathrm{ex}}$, and a universal necessary upper threshold $s_{\mathrm C}^{\mathrm{nec}}$, and prove $(s_{\mathrm P},s_{\mathrm C}^{\mathrm{ex}})
\subseteq \mathcal S
\subseteq
[s_{\mathrm P},s_{\mathrm C}^{\mathrm{nec}})$, where $\mathcal S$ is the set of speeds admitting positive extinction-to-coexistence waves. Existence follows from new lower solutions that remove previously imposed diffusion and compatibility restrictions. Nonexistence below $s_{\mathrm P}$ follows from a Sturm argument, while center-manifold analysis and a Riccati crossing argument yield the finite upper-speed obstruction. An explicit two-species wave beyond $s_{\mathrm C}^{\mathrm{ex}}$ shows that this constructive threshold is not the true maximal speed.
\end{abstract}

\maketitle

\tableofcontents

\section{Introduction and Main Results}
\label{sec 1}
Savannas and related plant--consumer ecosystems are spatially extended
systems shaped by the interaction of vegetation competition, dispersal,
herbivory, and environmental disturbance.  The balance between woody and
herbaceous vegetation is not determined by plant competition alone:
grazers and browsers respond to vegetation availability, while their
consumption feeds back into plant persistence and coexistence.  Herbivory
may also interact with fire through indirect feedbacks involving grass
biomass, fuel load, and woody vegetation
\cite{van2003effects,HoldoHoltFryxell2013}.  For a broader discussion of
spatial savanna models, we refer to \cite{sterl2026review}.

A basic spatial question is whether consumers can remain attached to an
advancing vegetation front.  Dispersal plays an important role in the
organization and resilience of savanna and forest distributions
\cite{goel2020dispersal}, and reaction--diffusion models have been used to
study spreading speeds and traveling waves in tree--grass systems
\cite{banasiak2023spreading,fazly2020analysis}.  Closely related
consumer--resource models arising from primary succession suggest that
herbivores may slow, arrest, or even reverse plant invasion
\cite{fagan2000trophic,fagan2005can}.  These observations lead naturally
to the following question:

\begin{quote}
For which propagation speeds can several competing plant populations and
their resource-dependent consumers jointly invade an initially depleted
region?
\end{quote}

\textcolor{black}{The present work is motivated by the four-species savanna model studied in
\cite{chen2025savanna}, but the passage to the present system is not merely
a dimensional extension.  In that work, the existence proof required
specific inequalities comparing the diffusion coefficients of the plants
and consumers, an additional diffusion-dependent restriction on the wave
speed, and several technical conditions introduced by the particular lower
solutions used in the construction.  As a result, existence was obtained
only in a parameter region constrained by the proof itself, rather than on
a speed interval determined solely by the coefficients of the original
system.}

\textcolor{black}{These restrictions reflect a genuine difficulty near the extinction state.
A plant lower profile has to play two different roles.  Near the leading
edge it must be very small and follow the correct exponential behavior.
Farther behind the front it must reach a positive level large enough to
sustain the corresponding consumer, whose reaction term has no linear
growth at extinction.  Forcing a single simple profile to satisfy both
requirements ties together its shape, the diffusion coefficients, and the
wave speed, and this is what produced the extra assumptions in the earlier
construction.}

\textcolor{black}{The present paper separates these two roles.  For each plant we use an
exponentially small profile near the leading edge and then join it to a
slowly increasing logistic profile whose eventual height can be chosen
independently.  For each consumer we use a Gaussian profile near the
leading edge and a small positive constant farther behind the front.  In
this way the consumer is required to become appreciable only after enough
plant biomass has developed.  This new lower-solution construction removes
the restrictions created by the previous proof: no special comparison
between plant and consumer diffusion coefficients is needed, no additional
diffusion-dependent speed bound is imposed, and no auxiliary matching
conditions remain.  Apart from the structural weak-interaction condition
and the explicit nonemptiness of the speed interval, the construction works
simultaneously for all components and for every speed in that interval.}

We formulate this mechanism for $N$ competing plant species and $N$
associated consumer populations:
\begin{equation}
\label{eq:pc-system}
\left\{
\begin{aligned}
\partial_t u_i
&=
d_i\partial_{xx}u_i
+r_i u_i
\left(
1-u_i-\kappa_{ii}v_i
-\sum_{\substack{j\in\N,\,j\neq i}}
\kappa_{ij}u_j
\right),
\\
\partial_t v_i
&=
D_i\partial_{xx}v_i
+e_i v_i(u_i-v_i),
\end{aligned}
\right.
\qquad i\in\N:=\{1,\ldots,N\}.
\end{equation}
Here $u_i=u_i(x,t)$ is the density of the $i$-th plant population,
and $v_i=v_i(x,t)$ is the density of its corresponding consumer.
The diffusion and growth parameters
$d_i$, $D_i$, $r_i$, $e_i$ are strictly positive.  The coefficient
$\kappa_{ii}>0$ measures the consumption pressure exerted by the
$i$-th consumer on its plant, whereas $\kappa_{ij}\geq0$, $j\neq i$,
measures the competitive effect of the $j$-th plant on the $i$-th
plant.

Writing
\begin{align*}
K=(\kappa_{ij})_{1\leq i,j\leq N},
\end{align*}
we impose the weak-interaction condition
\begin{equation}
\label{eq:weak-interaction}
\Theta
:=
\|K\|_\infty:=\max_{i\in\N}\sum_{j\in\N}\kappa_{ij}
<1.
\end{equation}
Under \eqref{eq:weak-interaction}, system \eqref{eq:pc-system} possesses
a unique strictly positive spatially homogeneous coexistence equilibrium
\begin{align*}
E^*=(u^*,v^*),
\qquad
v^*=u^*=(I+K)^{-1}\mathbf 1.
\end{align*}
Its existence and uniqueness are proved in Appendix~\ref{CE}.

\medskip

\textcolor{black}{The central analytical difficulty in \eqref{eq:pc-system} is the
degeneracy of the consumer dynamics at joint extinction.}  Indeed,
\begin{align*}
e_i v_i(u_i-v_i)=e_i u_i v_i-e_i v_i^2
\end{align*}
has no linear part at $(u_i,v_i)=(0,0)$.
\textcolor{black}{Thus, while the plants possess intrinsic low-density growth, the consumers
have no linear growth mechanism at the leading edge and can increase only
through their nonlinear coupling to the plants.  In the traveling-wave
ordinary differential equation this produces zero eigenvalues in the
consumer directions, so the extinction state is not hyperbolic there.
Consequently, the usual linear leading-edge analysis does not by itself
provide either the decay behavior of the consumer profiles or a standard
KPP-type lower-solution construction.}

\textcolor{black}{This degeneracy also produces a propagation mechanism fundamentally
different from the familiar Fisher--KPP picture.  In a scalar monostable
equation, admissible wave speeds typically form an unbounded half-line
above a minimal speed \cite{fisher1937wave,AronsonWeinberger1978}.  Here
the plant equations still impose a lower threshold, but the consumers,
which have no intrinsic linear growth near extinction, cannot remain
attached to a front moving arbitrarily fast.  They therefore impose a
finite upper obstruction.  The admissible speeds are consequently trapped
between a plant-driven lower threshold and a consumer-driven upper
threshold, giving a bounded propagation window.}

We seek traveling waves of the form
\begin{align*}
(u,v)(x,t)=(U,V)(\xi),
\qquad
\xi=x+st,
\end{align*}
where $s\in\mathbb{R}$.  The profiles solve
\begin{equation}
\label{eq:traveling-wave-system}
\left\{
\begin{aligned}
d_iU_i''-sU_i'
&+
r_iU_i
\left(
1-U_i-\kappa_{ii}V_i
-\sum_{\substack{j\in\N,\,j\neq i}}
\kappa_{ij}U_j
\right)
=0,
\\
D_iV_i''-sV_i'
&+
e_iV_i(U_i-V_i)
=0,
\end{aligned}
\right.
\qquad i\in\N,
\end{equation}
and we consider strictly positive waves connecting extinction to
coexistence:
\begin{equation}
\label{eq:wave-boundary-conditions}
\lim_{\xi\to-\infty}(U,V)(\xi)=(0,0),
\qquad
\lim_{\xi\to+\infty}(U,V)(\xi)=E^*.
\end{equation}

For each $i\in\N$, set
\begin{align*}
\sigma_i
:=
1-\sum_{j\in\N}\kappa_{ij}>0.
\end{align*}
Three characteristic speeds govern the propagation problem:
\begin{equation}
\label{eq:speed-thresholds}
s_{\mathrm P}
:=
\max_{i\in\N}2\sqrt{d_ir_i},
\qquad
s_{\mathrm C}^{\mathrm{ex}}
:=
\min_{i\in\N}2\sqrt{D_ie_i\sigma_i},
\qquad
s_{\mathrm C}^{\mathrm{nec}}
:=
\min_{i\in\N}2\sqrt{D_ie_i}.
\end{equation}
The first speed is determined by the linearized plant equations at the
extinction state.  The second arises from the constructive
upper--lower solution argument.  The third is a universal necessary
upper bound imposed by the degenerate consumer equations.

To formulate the main conclusion, define the admissible-speed set
\begin{align*}
\mathcal S
:=
\left\{
s\in\mathbb{R}:
\text{\eqref{eq:traveling-wave-system} admits a strictly positive
solution satisfying \eqref{eq:wave-boundary-conditions}}
\right\}.
\end{align*}
Our results establish the propagation-window bounds
\begin{equation}
\label{eq:master-window}
\boxed{
(s_{\mathrm P},s_{\mathrm C}^{\mathrm{ex}})
\subseteq
\mathcal S
\subseteq
[s_{\mathrm P},s_{\mathrm C}^{\mathrm{nec}})
}.
\end{equation}
Moreover, for an explicit two-species system,
\begin{equation}
\label{eq:intermediate-nonempty}
\mathcal S\cap
(s_{\mathrm C}^{\mathrm{ex}},s_{\mathrm C}^{\mathrm{nec}})
\neq\varnothing.
\end{equation}
Thus $s_{\mathrm C}^{\mathrm{ex}}$ is a sufficient threshold produced
by the construction, but it is not an intrinsic maximal speed of the
system.

We emphasize that the term \emph{propagation window} refers to the
rigorous inner and outer bounds in \eqref{eq:master-window}.  We do not
claim that $\mathcal S$ is necessarily an interval, nor do we determine
its complete structure inside $[s_{\mathrm C}^{\mathrm{ex}},s_{\mathrm C}^{\mathrm{nec}})$.

Our first result gives the inner existence bound.

\begin{theorem}[Existence inside the constructive window]
\label{thm:existence}
Assume that
\begin{equation}
\label{eq:nonempty-speed-window}
\Theta=\|K\|_\infty<1,
\qquad
s_{\mathrm P}<s_{\mathrm C}^{\mathrm{ex}}.
\end{equation}
Then, for every
\begin{align*}
s\in(s_{\mathrm P},s_{\mathrm C}^{\mathrm{ex}}),
\end{align*}
system~\eqref{eq:traveling-wave-system} admits a strictly positive
solution
\begin{align*}
(U,V)\in C^\infty(\mathbb R,\mathbb R^{2N})
\end{align*}
satisfying \eqref{eq:wave-boundary-conditions}.
\end{theorem}

\textcolor{black}{The proof uses ordered upper and lower solutions adapted to the absence of
linear consumer growth.  The key point is to let the plant lower profile
serve different purposes in different parts of the front.  It is
exponentially small near extinction, where the leading-edge behavior has
to be controlled, and farther behind it is continued by a slowly
increasing logistic profile whose limiting height can be chosen
independently.  The consumer lower profile is Gaussian near the leading
edge and is continued by a small positive constant once the plant density
has become sufficiently large.  This reflects the actual mechanism of the
system: the consumer does not grow on its own at low density, but only
after enough plant biomass is present.}

\textcolor{black}{The new construction removes all technical parameter restrictions that came
from the lower solutions in the earlier four-component argument.  In
particular, one no longer needs a prescribed ordering or separation
between plant and consumer diffusion coefficients, an extra speed bound
depending on such a comparison, or auxiliary inequalities used only to
match the profiles.  The lower and upper solutions can now be chosen
simultaneously for all components and for every}
\begin{align*}
s\in(s_{\mathrm P},s_{\mathrm C}^{\mathrm{ex}}).
\end{align*}
\textcolor{black}{Thus, even in the original four-component setting,
Theorem~\ref{thm:existence} substantially strengthens the
extinction-to-coexistence existence result of
\cite{chen2025savanna}; the extension to general $N$ is an additional
consequence rather than the sole source of novelty.}

The lower boundary is sharp in the following sense.

\begin{theorem}[Nonexistence below the plant threshold]
\label{thm:nonexistence-below-sm}
If
\begin{align*}
s<s_{\mathrm P}:=\max_{i\in\N}2\sqrt{d_ir_i},
\end{align*}
then system~\eqref{eq:traveling-wave-system} admits no strictly positive
solution satisfying \eqref{eq:wave-boundary-conditions}.
\end{theorem}

Theorems~\ref{thm:existence} and
\ref{thm:nonexistence-below-sm} imply
\begin{align*}
\inf\mathcal S=s_{\mathrm P}
\end{align*}
whenever $s_{\mathrm P}<s_{\mathrm C}^{\mathrm{ex}}$.  Thus
$s_{\mathrm P}$ is the sharp lower boundary of the admissible-speed
set in the infimum sense.  We do not claim here that a wave exists at
the endpoint $s=s_{\mathrm P}$.

The consumer equations yield an obstruction at the opposite end.

\begin{theorem}[Consumer-driven upper obstruction]
\label{thm:upper-speed-nonexistence}
If
\begin{align*}
s\geq
s_{\mathrm C}^{\mathrm{nec}}:=\min_{i\in\N}2\sqrt{D_ie_i},
\end{align*}
then system~\eqref{eq:traveling-wave-system} admits no strictly positive
solution satisfying \eqref{eq:wave-boundary-conditions}.
\end{theorem}

\textcolor{black}{The proof uses the degenerate leading-edge dynamics in an essential way,
and its main difficulty is to determine the decay of the consumer profiles
when the linearized system contains zero modes.  Exponential decay cannot
simply be read off from the linearization.  We first use a center-manifold
reduction together with exponential tracking to describe the possible
slow approach to extinction, and then use positivity of the traveling wave
to rule out every nonzero trajectory along the center directions.  This
forces every positive wave to decay exponentially at $-\infty$ and yields}
\begin{align*}
\frac{V_i'(\xi)}{V_i(\xi)}
\longrightarrow
\frac{s}{D_i}
\qquad
\text{as }\xi\to-\infty.
\end{align*}
\textcolor{black}{Once this leading-edge behavior is established, a Riccati equation for
$V_i'/V_i$, evaluated at its last crossing of $s/(2D_i)$, gives}
\begin{align*}
s<2\sqrt{D_ie_i}
\qquad
\text{for every }i\in\N.
\end{align*}
\textcolor{black}{The upper obstruction is therefore universal: unlike
$s_{\mathrm C}^{\mathrm{ex}}$, it does not involve the competition matrix
$K$.  It is a direct consequence of the fact that the consumers have no
linear growth at extinction.}

The interval
\begin{align*}
[s_{\mathrm C}^{\mathrm{ex}},s_{\mathrm C}^{\mathrm{nec}})
\end{align*}
is not merely an artifact of comparing a sufficient condition with a
necessary condition.  It contains genuine traveling waves.

\begin{proposition}[An explicit wave beyond the constructive threshold]
\label{prop}
There exist admissible coefficients for $N=2$, satisfying the
weak-interaction condition, for which
system~\eqref{eq:traveling-wave-system} admits a strictly positive
traveling wave with speed $s_{\mathrm{exact}}$ such that
\begin{align*}
s_{\mathrm P}
<
s_{\mathrm C}^{\mathrm{ex}}
<
s_{\mathrm{exact}}
<
s_{\mathrm C}^{\mathrm{nec}}.
\end{align*}
Consequently, $s_{\mathrm C}^{\mathrm{ex}}$ is not an intrinsic upper barrier for
positive extinction-to-coexistence traveling waves.
\end{proposition}

The explicit profile is obtained from a hyperbolic-tangent ansatz.
Exact traveling waves have previously been useful as structural probes
in competition--diffusion systems
\cite{hung2012exact,chen2012exact}. Here the explicit solution serves a
different purpose: it separates the threshold produced by a general
construction from the actual propagation capability of the system.
It shows that the intermediate regime is mathematically genuine and
that determining the complete admissible-speed set remains a
nontrivial open problem.

Traveling waves in predator--prey and consumer--resource systems have
been studied using topological, dynamical-systems, shooting, and
upper--lower solution methods; see, for example,
\cite{gardner1984existence,dunbar1986traveling,
huang2003existence,hsu2012existence,hsu2019existence}.
\textcolor{black}{The present work advances the preceding theory in three related directions.
First, we resolve the main existence difficulty created by the loss of the
linear consumer growth term.  The new lower solutions work for a general
$2N$-component system and require no extra comparison between plant and
consumer diffusion coefficients and no technical matching assumptions;
the only remaining conditions are the weak-interaction hypothesis and the
nonemptiness of the explicit speed interval.  This improvement is already
new for the original four-component model.  Second, we identify a new
consequence of the same linear degeneracy: a universal finite upper speed
for full extinction-to-coexistence waves.  Its proof requires a separate
leading-edge analysis to exclude slow center-direction decay before the
Riccati argument can be applied.  Third, combining this upper obstruction
with the sharp plant-driven lower threshold, existence on a full explicit
interval, and an exact wave beyond the constructive threshold yields the
bounded propagation-window structure \eqref{eq:master-window}.}

Before closing the introduction, we recall that the admissible-speed
set of a reaction--diffusion front depends strongly on the type of
connection under consideration.  For scalar monostable equations, and
for many monostable competition systems, traveling fronts typically
exist for every speed above a minimal threshold, so that the
admissible-speed set is a half-line
\cite{AronsonWeinberger1978, ChenHsiaoWang2024}.
By contrast, scalar bistable fronts typically propagate with a
uniquely selected speed
\cite{FifeMcLeod1977,ItoNinomiya2026}.
Ito and Ninomiya recently developed min--max characterizations of the
minimal and maximal speeds associated with two oppositely truncated
semi-front problems.  Under their hypotheses, these quantities
describe the infimum and supremum of the speed set of full bounded
fronts; in the classical bistable case they coincide with the unique
wave speed.  Their unbounded fronts again exhibit a one-sided
minimal-threshold structure
\cite{ItoNinomiya2026}.

Finite upper restrictions for full traveling waves are comparatively
rare.  Important examples arise in autocatalytic and generalized
Gray--Scott systems with linear or superlinear decay, for which the
admissible-speed set associated with a fixed upstream state is
nonempty but bounded
\cite{ChenQiZhang2016,ZhengChenQiZhou2018}.
Multiple propagation scales also occur in biological competition
models, for instance through distinct species-dependent spreading
speeds in free-boundary problems and through different invasion modes
in multi-species Cauchy problems
\cite{DuWu2018,Wu2019,JiangWu2026}.
These results concern different moving interfaces or
species-dependent spreading rates, rather than a bounded set of
admissible speeds for one fixed full heteroclinic connection.

\textcolor{black}{The present plant--consumer system exhibits a different mechanism: the
plant equations determine the sharp lower boundary, whereas the loss of
linear consumer growth at extinction produces a universal finite upper
obstruction.  Moreover, the constructive and necessary upper thresholds
are genuinely distinct, as shown by an explicit full front lying strictly
between them.  To the best of our knowledge, this combination of a general
open existence interval obtained without additional inequalities comparing
diffusion coefficients or technical profile-matching conditions, a
universal upper-speed obstruction, and an explicit wave beyond the
constructive threshold has not previously been established for a
high-dimensional plant--consumer reaction--diffusion system.}

\section{The traveling-wave system and generalized upper--lower solutions for $s_{\mathrm P}<s<s_{\mathrm C}^{\mathrm{ex}}$} \label{sec 2}

Let
\[
 \I:=\{1,\dots,N\}.
\]
For each $i\in\I$, consider
\begin{equation}
\label{eq:TW}
\left\{
\begin{aligned}
 d_i U_i''-sU_i'
 &+r_iU_i\left(1-U_i-\kappa_{ii}V_i
 -\sum_{\substack{j\in\I,\,j\ne i}}\kappa_{ij}U_j\right)=0,
 \\
 D_i V_i''-sV_i'
 &+e_iV_i(U_i-V_i)=0.
\end{aligned}
\right.
\end{equation}
We assume
\begin{equation}
\label{eq:positive-parameters}
 d_i,D_i,r_i,e_i>0,
 \qquad
 \kappa_{ij}\ge 0
 \quad (i,j\in\I).
\end{equation}
Set
\begin{equation}
\label{eq:rho-sigma}
 \rho_i:=\sum_{j\in\I}\kappa_{ij},
 \qquad
 \sigma_i:=1-\rho_i,
 \qquad
 \Theta:=\max_{i\in\I}\rho_i.
\end{equation}
Throughout we impose
\begin{equation}
\label{eq:weak-competition}
 \Theta<1.
\end{equation}
Hence
\begin{equation}
\label{eq:sigma-positive}
 0<\sigma_i\le 1
 \qquad (i\in\I).
\end{equation}
Define the characteristic speeds
\begin{equation}
\label{eq:speeds}
 s_{\mathrm P}:=\max_{i\in\I}2\sqrt{d_ir_i},
 \qquad
 s_{\mathrm C}^{\mathrm{ex}}
 :=\min_{i\in\I}2\sqrt{D_ie_i\sigma_i}.
\end{equation}

\subsection{Generalized upper and lower solutions}

We use the mixed ordering dictated by the signs of the coupling terms. A quadruple
\[
 (\overline U,\overline V,\underline U,\underline V)
\]
is called an ordered generalized upper--lower pair for \eqref{eq:TW} if the following properties hold.

\begin{enumerate}[label=(H\arabic*),leftmargin=12mm]
\item Each component is bounded, positive, and continuous on $\R$, and is $C^2$ away from a finite set of corner points.

\item The componentwise ordering
\begin{equation}
\label{eq:ordering-definition}
 0<\underline U_i\le \overline U_i,
 \qquad
 0<\underline V_i\le \overline V_i
 \qquad\text{on }\R
\end{equation}
holds for every $i\in\I$.

\item Away from the corner set,
\begin{align}
\label{eq:upper-U-def}
 d_i\overline U_i''-s\overline U_i'
 &+r_i\overline U_i
 \left(1-\overline U_i-\kappa_{ii}\underline V_i
 -\sum_{j\ne i}\kappa_{ij}\underline U_j\right)
 \le 0,
\end{align}
\begin{align}
\label{eq:lower-U-def}
 d_i\underline U_i''-s\underline U_i'
 &+r_i\underline U_i
 \left(1-\underline U_i-\kappa_{ii}\overline V_i
 -\sum_{j\ne i}\kappa_{ij}\overline U_j\right)
 \ge 0,
\end{align}
\begin{align}
\label{eq:upper-V-def}
 D_i\overline V_i''-s\overline V_i'
 +e_i\overline V_i(\overline U_i-\overline V_i)
 \le 0,
\end{align}
and
\begin{align}
\label{eq:lower-V-def}
 D_i\underline V_i''-s\underline V_i'
 +e_i\underline V_i(\underline U_i-\underline V_i)
 \ge 0.
\end{align}

\item At every corner $\xi_0$, upper solutions have downward derivative jumps and lower solutions have upward derivative jumps:
\begin{equation}
\label{eq:corner-definition}
 \overline U_i'(\xi_0^+)\le \overline U_i'(\xi_0^-),
 \qquad
 \overline V_i'(\xi_0^+)\le \overline V_i'(\xi_0^-),
\end{equation}
\begin{equation}
\label{eq:corner-definition-lower}
 \underline U_i'(\xi_0^-)\le \underline U_i'(\xi_0^+),
 \qquad
 \underline V_i'(\xi_0^-)\le \underline V_i'(\xi_0^+).
\end{equation}
At a corner belonging to another component, the relevant component is smooth, so the corresponding inequality holds with equality.
\end{enumerate}

\begin{lemma}
\label{lem:ordered-barriers}
Assume \eqref{eq:positive-parameters} and \eqref{eq:weak-competition}. Fix
\begin{equation}
\label{eq:s-fixed}
 s\in(s_{\mathrm P},s_{\mathrm C}^{\mathrm{ex}}).
\end{equation}
Then there exists an ordered generalized upper--lower pair
\[
 (\overline U,\overline V,\underline U,\underline V)
\]
for \eqref{eq:TW}. Moreover,
\begin{equation}
\label{eq:left-limits}
 \lim_{\xi\to-\infty}
 (\overline U,\overline V,\underline U,\underline V)(\xi)=0,
\end{equation}
while, for every $i\in\I$,
\begin{equation}
\label{eq:right-bounds}
 0<\lim_{\xi\to+\infty}\underline U_i(\xi)<\sigma_i,
 \qquad
 0<\lim_{\xi\to+\infty}\underline V_i(\xi)<1,
\end{equation}
and
\begin{equation}
\label{eq:upper-right-limits}
 \lim_{\xi\to+\infty}\overline U_i(\xi)
 =\lim_{\xi\to+\infty}\overline V_i(\xi)=1.
\end{equation}
\end{lemma}

The proof is constructive. We first define the upper solutions, then construct the plant lower solution from a left bump and a logistic bridge, and finally construct the Gaussian consumer lower solution. The last section verifies that all restrictions are simultaneously compatible.

\begin{lemma}[Existence between ordered generalized upper and lower solutions]
\label{lem:existence-between-barriers}
Let $s>0$. Assume that system~\eqref{eq:TW} admits an ordered generalized
upper--lower pair
\[
(\overline U,\overline V,\underline U,\underline V)
\]
in the sense of the preceding definition. Then there exists a solution
\[
(U,V)\in C^\infty(\mathbb R,\mathbb R^{2N})
\]
of \eqref{eq:TW} such that, for every $i\in\I$ and $\xi\in\mathbb R$,
\[
\underline U_i(\xi)\le U_i(\xi)\le\overline U_i(\xi),
\qquad
\underline V_i(\xi)\le V_i(\xi)\le\overline V_i(\xi).
\]
In particular, if all lower components are strictly positive, then $(U,V)$
is strictly positive on $\mathbb R$.
\end{lemma}

\begin{proof}
The result follows from the Schauder fixed-point construction for
generalized upper and lower solutions developed in
\cite[Sections~4.1--4.2 and Appendix~A]{chen2025savanna};
see also \cite[Theorem~2.2]{ma2001traveling} for the underlying
fixed-point framework. Although the construction in \cite{chen2025savanna} is written for a
four-component system, the argument is componentwise and extends
without change to any finite number of components. In the present
setting, the mixed quasimonotonicity of the reaction terms and the
generalized derivative-jump conditions imply that the closed order
interval determined by $(\underline U,\underline V)$ and $(\overline U,\overline V)$ is invariant under the associated Green operator. The Schauder
fixed-point theorem therefore yields a solution
\[
(U,V)\in C^2(\mathbb R,\mathbb R^{2N})
\]
satisfying
\[
(\underline U,\underline V)
\leq
(U,V)
\leq
(\overline U,\overline V).
\]
Since the right-hand side of \eqref{eq:TW} is polynomial, standard
ODE regularity implies
\[
(U,V)\in C^\infty(\mathbb R,\mathbb R^{2N}).
\]
Strict positivity follows from the strict positivity of the lower
profiles.
\end{proof}

\subsection{Characteristic roots and upper solutions}

Fix $i\in\I$ and define
\begin{equation}
\label{eq:pi}
 p_i(z):=d_i z^2-sz+r_i.
\end{equation}
Since $s>s_{\mathrm P}$,
\begin{equation}
\label{eq:discriminant-positive}
 s^2-4d_ir_i>0.
\end{equation}
Thus $p_i$ has two distinct positive roots
\begin{equation}
\label{eq:lambda-pm}
 \lambda_i^-:=\frac{s-\sqrt{s^2-4d_ir_i}}{2d_i},
 \qquad
 \lambda_i^+:=\frac{s+\sqrt{s^2-4d_ir_i}}{2d_i},
\end{equation}
with
\begin{equation}
\label{eq:lambda-order}
 0<\lambda_i^-<\lambda_i^+,
 \qquad
 p_i(\lambda_i^-)=p_i(\lambda_i^+)=0.
\end{equation}
For the consumer upper profile, set
\begin{equation}
\label{eq:nu}
 \nu_i:=\min\left\{\lambda_i^-,\frac{s}{D_i}\right\}>0.
\end{equation}
Define
\begin{equation}
\label{eq:upper-U}
 \overline U_i(\xi):=
 \begin{cases}
 e^{\lambda_i^-\xi},&\xi\le0,\\[1mm]
 1,&\xi\ge0,
 \end{cases}
\end{equation}
and
\begin{equation}
\label{eq:upper-V}
 \overline V_i(\xi):=
 \begin{cases}
 e^{\nu_i\xi},&\xi\le0,\\[1mm]
 1,&\xi\ge0.
 \end{cases}
\end{equation}
Both are positive, bounded, continuous, and piecewise $C^2$. At $\xi=0$,
\begin{equation}
\label{eq:upper-corners}
 \overline U_i'(0^+)=0\le \lambda_i^-=\overline U_i'(0^-),
 \qquad
 \overline V_i'(0^+)=0\le \nu_i=\overline V_i'(0^-).
\end{equation}

\begin{lemma}[Upper inequalities]
\label{lem:upper}
The functions \eqref{eq:upper-U}--\eqref{eq:upper-V} satisfy
\eqref{eq:upper-U-def} and \eqref{eq:upper-V-def} for every ordered positive lower pair satisfying
\[
 0<\underline U_j\le\overline U_j,
 \qquad
 0<\underline V_i\le\overline V_i.
\]
\end{lemma}

\begin{proof}
For $\xi\le0$, using $p_i(\lambda_i^-)=0$ gives
\begin{align*}
&d_i\overline U_i''-s\overline U_i'
+r_i\overline U_i
\left(1-\overline U_i-\kappa_{ii}\underline V_i
-\sum_{j\ne i}\kappa_{ij}\underline U_j\right)=-r_i\overline U_i
\left(\overline U_i+\kappa_{ii}\underline V_i
+\sum_{j\ne i}\kappa_{ij}\underline U_j\right)
\le0.
\end{align*}
For $\xi\ge0$, $\overline U_i=1$, hence
\begin{align*}
&d_i\overline U_i''-s\overline U_i'
+r_i\overline U_i
\left(1-\overline U_i-\kappa_{ii}\underline V_i
-\sum_{j\ne i}\kappa_{ij}\underline U_j\right)=-r_i\left(\kappa_{ii}\underline V_i
+\sum_{j\ne i}\kappa_{ij}\underline U_j\right)
\le0.
\end{align*}
This proves \eqref{eq:upper-U-def}.

For $\xi\le0$, $\nu_i\le s/D_i$ gives
\[
 D_i\nu_i^2-s\nu_i=\nu_i(D_i\nu_i-s)\le0.
\]
Also $\nu_i\le\lambda_i^-$, and since $\xi\le0$,
\[
 \overline U_i(\xi)=e^{\lambda_i^-\xi}
 \le e^{\nu_i\xi}=\overline V_i(\xi).
\]
Therefore
\begin{align*}
&D_i\overline V_i''-s\overline V_i'
+e_i\overline V_i(\overline U_i-\overline V_i)=\nu_i(D_i\nu_i-s)e^{\nu_i\xi}
+e_ie^{\nu_i\xi}
\left(e^{\lambda_i^-\xi}-e^{\nu_i\xi}\right)
\le0.
\end{align*}
For $\xi\ge0$, $\overline U_i=\overline V_i=1$, so the left-hand side vanishes. This proves \eqref{eq:upper-V-def}.
\end{proof}

\subsection{Plant lower solution: bump plus logistic bridge}
Next, let us construct the lower solution for plants.
\noindent
{\bf Choice of the bump exponent.} For $i\in\I$, let
\begin{equation}
\label{eq:Ji}
 \I_i^+:=\{j\in\I:\ j\ne i,\ \kappa_{ij}>0\}.
\end{equation}
Define
\begin{equation}
\label{eq:Mi}
 M_i:=\min\left\{
 2,
 \frac{\lambda_i^+}{\lambda_i^-},
 1+\frac{\nu_i}{\lambda_i^-},
 \min_{j\in\I_i^+}\left(1+\frac{\lambda_j^-}{\lambda_i^-}\right)
 \right\},
\end{equation}
where the last minimum is omitted if $\I_i^+=\varnothing$. Every number in \eqref{eq:Mi} is strictly larger than $1$, hence
\begin{equation}
\label{eq:Mi-greater-one}
 M_i>1.
\end{equation}
Choose, for example,
\begin{equation}
\label{eq:phi-explicit}
 \varphi_i:=\frac{1+M_i}{2}.
\end{equation}
Then
\begin{equation}
\label{eq:phi-range}
 1<\varphi_i<M_i.
\end{equation}
In particular,
\begin{equation}
\label{eq:phi-root-range}
 \lambda_i^-<\varphi_i\lambda_i^-<\lambda_i^+.
\end{equation}
Since $p_i$ is negative strictly between its two roots,
\begin{equation}
\label{eq:ci}
 c_i:=-p_i(\varphi_i\lambda_i^-)>0.
\end{equation}

The other restrictions in \eqref{eq:Mi} imply
\begin{equation}
\label{eq:exponent-relations}
 \varphi_i\lambda_i^-<2\lambda_i^-,
 \qquad
 \varphi_i\lambda_i^-<\lambda_i^-+\nu_i,
\end{equation}
and, whenever $j\in\I_i^+$,
\begin{equation}
\label{eq:cross-exponent-relation}
 \varphi_i\lambda_i^-<\lambda_i^-+\lambda_j^-.
\end{equation}
These inequalities are the exact exponent comparisons needed below.

\noindent
{\bf Choice and geometry of the bump.} Choose
\begin{equation}
\label{eq:Ai-choice}
 A_i>\max\left\{1,\frac{r_i(2-\sigma_i)}{c_i}\right\}.
\end{equation}
Define
\begin{equation}
\label{eq:fi}
 f_i(\xi):=e^{\lambda_i^-\xi}
 -A_ie^{\varphi_i\lambda_i^-\xi}.
\end{equation}
The unique zero of $f_i$ is
\begin{equation}
\label{eq:zi}
 z_i:=-\frac{\log A_i}{(\varphi_i-1)\lambda_i^-}<0.
\end{equation}
The unique critical point of $f_i$ is
\begin{equation}
\label{eq:ami}
 a_i^{\mathrm m}:=-\frac{\log(\varphi_iA_i)}
 {(\varphi_i-1)\lambda_i^-},
\end{equation}
and
\begin{equation}
\label{eq:am-z-order}
 a_i^{\mathrm m}<z_i<0.
\end{equation}
Indeed,
\begin{equation}
\label{eq:fi-prime}
 f_i'(\xi)
 =\lambda_i^-e^{\lambda_i^-\xi}
 \left(1-\varphi_iA_ie^{(\varphi_i-1)\lambda_i^-\xi}\right).
\end{equation}
Consequently,
\begin{equation}
\label{eq:fi-geometry}
 \begin{cases}
 f_i(\xi)>0,&\xi<z_i,\\
 f_i'(\xi)>0,&\xi<a_i^{\mathrm m},\\
 f_i'(\xi)<0,&a_i^{\mathrm m}<\xi<z_i.
 \end{cases}
\end{equation}
The maximum value is positive and is given explicitly by
\begin{equation}
\label{eq:fi-max}
 f_i(a_i^{\mathrm m})
 =\frac{\varphi_i-1}{\varphi_i}
 \left(\frac{1}{\varphi_iA_i}\right)^{1/(\varphi_i-1)}>0.
\end{equation}

\noindent
{\bf Choice of the two positive tail levels.} Since $s<s_{\mathrm C}^{\mathrm{ex}}$, for every $i\in\I$,
\begin{equation}
\label{eq:theta-sigma}
 \theta_i:=\frac{s^2}{4D_ie_i}<\sigma_i.
\end{equation}
It is useful to make the choice explicit:
\begin{equation}
\label{eq:delta-explicit}
 \delta_i^0:=\frac{2\theta_i+\sigma_i}{3},
 \qquad
 \delta_i^\infty:=\frac{\theta_i+2\sigma_i}{3}.
\end{equation}
Then
\begin{equation}
\label{eq:delta-order}
 \theta_i<\delta_i^0<\delta_i^\infty<\sigma_i.
\end{equation}
The level $\delta_i^0$ will be used in the consumer inequality. The larger level $\delta_i^\infty$ will be the limiting value of the plant bridge.

\noindent
{\bf Choice of the junction point.} Define
\begin{equation}
\label{eq:mi}
 m_i:=\frac12\min\left\{\delta_i^0,f_i(a_i^{\mathrm m})\right\}.
\end{equation}
Then
\begin{equation}
\label{eq:mi-bounds}
 0<m_i<\delta_i^0,
 \qquad
 0<m_i<f_i(a_i^{\mathrm m}).
\end{equation}
Because $f_i$ is continuous and strictly decreasing from $f_i(a_i^{\mathrm m})$ to $0$ on $(a_i^{\mathrm m},z_i)$, there exists a unique
\begin{equation}
\label{eq:ai-location}
 a_i\in(a_i^{\mathrm m},z_i)
\end{equation}
such that
\begin{equation}
\label{eq:ai-def}
 f_i(a_i)=m_i.
\end{equation}
It follows that
\begin{equation}
\label{eq:ai-properties}
 a_i<0,
 \qquad
 f_i'(a_i)<0,
 \qquad
 f_i(\xi)>0\quad(\xi\le a_i).
\end{equation}
Notice that $m_i$ is chosen below the bump maximum automatically. No consumer-compatible plateau is required to lie below that maximum.

\noindent
{\bf Choice of the logistic rate.} Set
\begin{equation}
\label{eq:gap-ri}
 R_i:=r_i(\sigma_i-\delta_i^\infty)>0.
\end{equation}
Choose $\gamma_i>0$ such that
\begin{equation}
\label{eq:gamma-choice}
 0<\gamma_i<\gamma_i^*:=
 \min\left\{
 \frac{\lambda_i^-}{\delta_i^\infty},
 \frac{R_i}{2s\delta_i^\infty},
 \frac{1}{\delta_i^\infty}\sqrt{\frac{R_i}{2d_i}}
 \right\}.
\end{equation}
This interval is nonempty because every entry in the minimum is positive. The choice implies
\begin{equation}
\label{eq:gamma-order-condition}
 \gamma_i\delta_i^\infty<\lambda_i^-,
\end{equation}
\begin{equation}
\label{eq:gamma-linear-bound}
 s\gamma_i\delta_i^\infty<\frac{R_i}{2},
\end{equation}
and
\begin{equation}
\label{eq:gamma-quadratic-bound}
 d_i\gamma_i^2(\delta_i^\infty)^2<\frac{R_i}{2}.
\end{equation}
Hence
\begin{equation}
\label{eq:gamma-main}
 s\gamma_i\delta_i^\infty
 +d_i\gamma_i^2(\delta_i^\infty)^2
 <R_i=r_i(\sigma_i-\delta_i^\infty).
\end{equation}

\subsection{Definition of the bridge and the plant lower solution}

Let $h_i$ solve
\begin{equation}
\label{eq:logistic-ode}
 h_i'=\gamma_i h_i(\delta_i^\infty-h_i),
 \qquad
 h_i(a_i)=m_i.
\end{equation}
Its explicit formula is
\begin{equation}
\label{eq:hi-explicit}
 h_i(\xi)
 =\frac{\delta_i^\infty}
 {1+\left(\frac{\delta_i^\infty}{m_i}-1\right)
 e^{-\gamma_i\delta_i^\infty(\xi-a_i)}}.
\end{equation}
For $\xi\ge a_i$,
\begin{equation}
\label{eq:hi-properties}
 m_i\le h_i(\xi)<\delta_i^\infty,
 \qquad
 h_i'(\xi)>0,
 \qquad
 \lim_{\xi\to+\infty}h_i(\xi)=\delta_i^\infty.
\end{equation}
Define
\begin{equation}
\label{eq:lower-U}
 \underline U_i(\xi):=
 \begin{cases}
 f_i(\xi),&\xi\le a_i,\\[1mm]
 h_i(\xi),&\xi\ge a_i.
 \end{cases}
\end{equation}
By \eqref{eq:ai-def} and \eqref{eq:logistic-ode},
\begin{equation}
\label{eq:lower-U-continuity-corner}
 f_i(a_i)=h_i(a_i)=m_i,
 \qquad
 f_i'(a_i)<0<h_i'(a_i).
\end{equation}
Thus $\underline U_i$ is continuous, positive, piecewise $C^2$, and satisfies the required upward derivative jump at $a_i$.

\begin{lemma}[Ordering of the plant profiles]
\label{lem:plant-ordering}
For every $i\in\I$,
\begin{equation}
\label{eq:plant-ordering}
 0<\underline U_i(\xi)<\overline U_i(\xi)
 \qquad (\xi\in\R).
\end{equation}
\end{lemma}

\begin{proof}
For $\xi\le a_i$, positivity follows from \eqref{eq:ai-properties}, and
\[
 f_i(\xi)
 =e^{\lambda_i^-\xi}-A_ie^{\varphi_i\lambda_i^-\xi}
 <e^{\lambda_i^-\xi}
 =\overline U_i(\xi).
\]
For $a_i\le\xi\le0$, define
\begin{equation}
\label{eq:ratio-Ri}
 \mathcal R_i(\xi):=\frac{h_i(\xi)}{e^{\lambda_i^-\xi}}.
\end{equation}
Then
\begin{align}
\label{eq:ratio-derivative}
 \frac{\mathcal R_i'}{\mathcal R_i}
 &=\frac{h_i'}{h_i}-\lambda_i^-
 =\gamma_i(\delta_i^\infty-h_i)-\lambda_i^-
 \le\gamma_i\delta_i^\infty-\lambda_i^-<0
\end{align}
by \eqref{eq:gamma-order-condition}. Moreover,
\[
 \mathcal R_i(a_i)
 =\frac{f_i(a_i)}{e^{\lambda_i^-a_i}}<1.
\]
Hence $\mathcal R_i(\xi)<1$ for $a_i\le\xi\le0$, and therefore
\[
 h_i(\xi)<e^{\lambda_i^-\xi}=\overline U_i(\xi).
\]
Finally, for $\xi\ge0$,
\[
 0<h_i(\xi)<\delta_i^\infty<\sigma_i\le1=\overline U_i(\xi).
\]
This proves \eqref{eq:plant-ordering}.
\end{proof}

\begin{lemma}[Lower plant differential inequality]
\label{lem:lower-plant}
The profile \eqref{eq:lower-U} satisfies \eqref{eq:lower-U-def} away from $a_i$.
\end{lemma}

\begin{proof}
We treat the two pieces separately.

\emph{Step 1: the left bump $\xi<a_i$.}
Since $p_i(\lambda_i^-)=0$ and $c_i=-p_i(\varphi_i\lambda_i^-)$,
\begin{equation}
\label{eq:fi-linear}
 d_if_i''-sf_i'+r_if_i
 =A_ic_i e^{\varphi_i\lambda_i^-\xi}.
\end{equation}
Also $0<f_i(\xi)<e^{\lambda_i^-\xi}$ for $\xi\le a_i<0$. By \eqref{eq:exponent-relations}, and because $\xi\le0$,
\begin{equation}
\label{eq:fi-square-estimate}
 f_i(\xi)^2
 \le e^{2\lambda_i^-\xi}
 \le e^{\varphi_i\lambda_i^-\xi},
\end{equation}
\begin{equation}
\label{eq:fi-Vbar-estimate}
 f_i(\xi)\overline V_i(\xi)
 \le e^{(\lambda_i^-+\nu_i)\xi}
 \le e^{\varphi_i\lambda_i^-\xi},
\end{equation}
and, whenever $j\in\I_i^+$,
\begin{equation}
\label{eq:fi-Ujbar-estimate}
 f_i(\xi)\overline U_j(\xi)
 \le e^{(\lambda_i^-+\lambda_j^-)\xi}
 \le e^{\varphi_i\lambda_i^-\xi}.
\end{equation}
Terms with $\kappa_{ij}=0$ contribute nothing. Therefore
\begin{align}
\label{eq:nonlinear-bump-bound}
&f_i\left(f_i+\kappa_{ii}\overline V_i
+\sum_{j\ne i}\kappa_{ij}\overline U_j\right)\le
\left(1+\kappa_{ii}+\sum_{j\ne i}\kappa_{ij}\right)
 e^{\varphi_i\lambda_i^-\xi}
=(2-\sigma_i)e^{\varphi_i\lambda_i^-\xi}.
\end{align}
Combining \eqref{eq:fi-linear} and \eqref{eq:nonlinear-bump-bound},
\begin{align*}
&d_if_i''-sf_i'
+r_if_i\left(1-f_i-\kappa_{ii}\overline V_i
-\sum_{j\ne i}\kappa_{ij}\overline U_j\right)\ge
\left[A_ic_i-r_i(2-\sigma_i)\right]
 e^{\varphi_i\lambda_i^-\xi}>0
\end{align*}
by \eqref{eq:Ai-choice}.

\emph{Step 2: the logistic bridge $\xi>a_i$.}
Since every upper component is bounded above by $1$,
\begin{align}
\label{eq:bridge-reaction-lower}
1-h_i-\kappa_{ii}\overline V_i
-\sum_{j\ne i}\kappa_{ij}\overline U_j
&\ge1-h_i-\sum_{j\in\I}\kappa_{ij}
=\sigma_i-h_i.
\end{align}
Differentiating \eqref{eq:logistic-ode},
\begin{equation}
\label{eq:hi-second}
 h_i''=\gamma_i^2h_i(\delta_i^\infty-h_i)
 (\delta_i^\infty-2h_i).
\end{equation}
Because $0<h_i<\delta_i^\infty$,
\begin{equation}
\label{eq:hi-second-coarse}
 h_i''\ge-\gamma_i^2(\delta_i^\infty)^2h_i,
\end{equation}
\begin{equation}
\label{eq:hi-first-coarse}
 -h_i'\ge-\gamma_i\delta_i^\infty h_i,
\end{equation}
and
\begin{equation}
\label{eq:reaction-gap}
 \sigma_i-h_i\ge\sigma_i-\delta_i^\infty.
\end{equation}
Using \eqref{eq:bridge-reaction-lower}--\eqref{eq:reaction-gap},
\begin{align*}
&d_ih_i''-sh_i'
+r_ih_i\left(1-h_i-\kappa_{ii}\overline V_i
-\sum_{j\ne i}\kappa_{ij}\overline U_j\right)\ge h_i\left[
 r_i(\sigma_i-\delta_i^\infty)
 -s\gamma_i\delta_i^\infty
 -d_i\gamma_i^2(\delta_i^\infty)^2
 \right]>0
\end{align*}
by \eqref{eq:gamma-main}. This proves the lower plant inequality on both smooth pieces.
\end{proof}

\subsection{Consumer lower solution: Gaussian tail plus constant} Let us construct the lower solution for consumers. 

\noindent
{\bf The bridge crossing point.} Since $h_i(a_i)=m_i<\delta_i^0$ and
\[
 \lim_{\xi\to+\infty}h_i(\xi)=\delta_i^\infty>\delta_i^0,
\]
and since $h_i$ is strictly increasing, there exists a unique $b_i>a_i$ such that
\begin{equation}
\label{eq:bi-def}
 h_i(b_i)=\delta_i^0.
\end{equation}
The point is explicitly
\begin{equation}
\label{eq:bi-explicit}
 b_i=a_i+\frac{1}{\gamma_i\delta_i^\infty}
 \log\left(
 \frac{\delta_i^0(\delta_i^\infty-m_i)}
 {m_i(\delta_i^\infty-\delta_i^0)}
 \right).
\end{equation}
The logarithm is positive because
\[
 \delta_i^0(\delta_i^\infty-m_i)
 -m_i(\delta_i^\infty-\delta_i^0)
 =\delta_i^\infty(\delta_i^0-m_i)>0.
\]
Therefore $b_i>a_i$, and
\begin{equation}
\label{eq:h-lower-after-b}
 h_i(\xi)\ge\delta_i^0
 \qquad(\xi\ge b_i).
\end{equation}

\noindent
{\bf Choice of the Gaussian height and width.} Define the positive gap
\begin{equation}
\label{eq:Gi}
 G_i:=e_i\delta_i^0-\frac{s^2}{4D_i}>0,
\end{equation}
where positivity follows from $\delta_i^0>\theta_i=s^2/(4D_ie_i)$. Choose explicitly
\begin{equation}
\label{eq:vi-choice}
 v_i:=\frac12\min\left\{1,\delta_i^0,\frac{G_i}{4e_i}\right\}>0,
\end{equation}
and
\begin{equation}
\label{eq:beta-choice}
 \beta_i:=\frac{G_i}{16D_i}>0.
\end{equation}
Then
\begin{equation}
\label{eq:vi-basic}
 v_i<1,
 \qquad
 v_i<\delta_i^0,
 \qquad
 e_iv_i\le\frac{G_i}{8},
\end{equation}
and
\begin{equation}
\label{eq:beta-basic}
 2D_i\beta_i=\frac{G_i}{8}.
\end{equation}
Consequently,
\begin{align}
\label{eq:consumer-gap}
&-2D_i\beta_i-\frac{s^2}{4D_i}
+e_i(\delta_i^0-v_i)=G_i-2D_i\beta_i-e_iv_i
\ge G_i-\frac{G_i}{8}-\frac{G_i}{8}
=\frac{3G_i}{4}>0.
\end{align}

\noindent
{\bf Choice of the Gaussian center.} For $y\in\R$, define
\begin{equation}
\label{eq:Qi}
 Q_i(y):=4D_i\beta_i^2y^2+2s\beta_i y
 -2D_i\beta_i-e_iv_i.
\end{equation}
Choose
\begin{equation}
\label{eq:Li-choice}
\begin{aligned}
 L_i:=1+\max\Bigg\{
 &\frac{s}{D_i\beta_i},
 \sqrt{\frac{2D_i\beta_i+e_iv_i}{2D_i\beta_i^2}},
 -b_i,
 \frac{\nu_i}{2\beta_i}-b_i,
 0
 \Bigg\}.
\end{aligned}
\end{equation}
Then
\begin{equation}
\label{eq:L-properties}
 L_i>\frac{s}{D_i\beta_i}>
 \frac{s}{4D_i\beta_i},
 \qquad
 b_i+L_i>0,
 \qquad
 b_i+L_i>\frac{\nu_i}{2\beta_i}.
\end{equation}
Moreover, because $L_i>s/(D_i\beta_i)$,
\[
 2s\beta_iL_i<2D_i\beta_i^2L_i^2.
\]
Hence
\begin{align*}
 Q_i(-L_i)
 &=4D_i\beta_i^2L_i^2-2s\beta_iL_i
 -2D_i\beta_i-e_iv_i>2D_i\beta_i^2L_i^2-2D_i\beta_i-e_iv_i>0
\end{align*}
by the second lower bound in \eqref{eq:Li-choice}.

Set
\begin{equation}
\label{eq:xiVi}
 \xi_{V_i}:=b_i+L_i.
\end{equation}
Then
\begin{equation}
\label{eq:xiVi-properties}
 \xi_{V_i}>0,
 \qquad
 \xi_{V_i}>b_i,
 \qquad
 2\beta_i\xi_{V_i}>\nu_i,
 \qquad
 Q_i(b_i-\xi_{V_i})=Q_i(-L_i)>0.
\end{equation}

\noindent
{\bf Definition of the consumer lower profile.} Define
\begin{equation}
\label{eq:lower-V}
 \underline V_i(\xi):=
 \begin{cases}
 v_ie^{-\beta_i(\xi-\xi_{V_i})^2},
 &\xi\le\xi_{V_i},\\[1mm]
 v_i,&\xi\ge\xi_{V_i}.
 \end{cases}
\end{equation}
This function is positive, bounded, continuous, and piecewise $C^2$. At the junction,
\begin{equation}
\label{eq:lower-V-corner}
 \underline V_i'(\xi_{V_i}^-)=0
 =\underline V_i'(\xi_{V_i}^+),
\end{equation}
so the lower corner condition holds with equality.

\begin{lemma}[Ordering of the consumer profiles]
\label{lem:consumer-ordering}
For every $i\in\I$,
\begin{equation}
\label{eq:consumer-ordering}
 0<\underline V_i(\xi)<\overline V_i(\xi)
 \qquad(\xi\in\R).
\end{equation}
\end{lemma}

\begin{proof}
For $\xi\ge0$,
\[
 0<\underline V_i(\xi)\le v_i<1=\overline V_i(\xi).
\]
For $\xi\le0$, the Gaussian formula in \eqref{eq:lower-V} is valid because $\xi_{V_i}>0$. Define
\begin{equation}
\label{eq:Fi}
 F_i(\xi):=\log\frac{\underline V_i(\xi)}{\overline V_i(\xi)}
 =\log v_i-\beta_i(\xi-\xi_{V_i})^2-\nu_i\xi.
\end{equation}
Then
\begin{align}
\label{eq:Fi-prime}
 F_i'(\xi)
 &=2\beta_i(\xi_{V_i}-\xi)-\nu_i
 \ge2\beta_i\xi_{V_i}-\nu_i>0
 \qquad(\xi\le0)
\end{align}
by \eqref{eq:xiVi-properties}. Thus $F_i$ is increasing on $(-\infty,0]$, and
\begin{equation}
\label{eq:Fi-negative}
 F_i(\xi)\le F_i(0)
 =\log v_i-\beta_i\xi_{V_i}^2<0.
\end{equation}
Therefore $\underline V_i<\overline V_i$ on $(-\infty,0]$ as well.
\end{proof}

\begin{lemma}[Lower consumer differential inequality]
\label{lem:lower-consumer}
The profile \eqref{eq:lower-V} satisfies \eqref{eq:lower-V-def} away from $\xi_{V_i}$.
\end{lemma}

\begin{proof}
We divide the real line into three regions.

\emph{Region I: $\xi>\xi_{V_i}$.}
Here $\underline V_i=v_i$. Since $\xi_{V_i}>b_i$, \eqref{eq:h-lower-after-b} gives $\underline U_i=h_i\ge\delta_i^0$. Hence
\begin{align*}
&D_i\underline V_i''-s\underline V_i'
+e_i\underline V_i(\underline U_i-\underline V_i)=e_iv_i(\underline U_i-v_i)
\ge e_iv_i(\delta_i^0-v_i)>0.
\end{align*}

\emph{Region II: $b_i\le\xi<\xi_{V_i}$.}
Put
\begin{equation}
\label{eq:y-def}
 y:=\xi-\xi_{V_i}\le0.
\end{equation}
Then
\begin{equation}
\label{eq:gaussian-derivatives}
 \underline V_i'=-2\beta_i y\underline V_i,
 \qquad
 \underline V_i''=(4\beta_i^2y^2-2\beta_i)\underline V_i.
\end{equation}
Consequently,
\begin{align}
\label{eq:consumer-bracket}
&D_i\underline V_i''-s\underline V_i'
+e_i\underline V_i(\underline U_i-\underline V_i)=\underline V_i\left[
4D_i\beta_i^2y^2+2s\beta_i y-2D_i\beta_i
+e_i(\underline U_i-\underline V_i)
\right].
\end{align}
Completing the square,
\begin{equation}
\label{eq:complete-square}
 4D_i\beta_i^2y^2+2s\beta_i y
 =4D_i\beta_i^2
 \left(y+\frac{s}{4D_i\beta_i}\right)^2
 -\frac{s^2}{4D_i}
 \ge-\frac{s^2}{4D_i}.
\end{equation}
Also, by \eqref{eq:h-lower-after-b} and \eqref{eq:lower-V},
\[
 \underline U_i\ge\delta_i^0,
 \qquad
 \underline V_i\le v_i.
\]
Therefore \eqref{eq:consumer-bracket} and \eqref{eq:consumer-gap} imply
\begin{align*}
&D_i\underline V_i''-s\underline V_i'
+e_i\underline V_i(\underline U_i-\underline V_i)\ge\underline V_i
\left[-2D_i\beta_i-\frac{s^2}{4D_i}
+e_i(\delta_i^0-v_i)\right]>0.
\end{align*}

\emph{Region III: $\xi<b_i$.}
Again set $y=\xi-\xi_{V_i}$. Since $\xi<b_i$ and $\xi_{V_i}=b_i+L_i$,
\begin{equation}
\label{eq:y-left}
 y<-L_i.
\end{equation}
Using only $\underline U_i>0$ and $\underline V_i\le v_i$, the bracket in \eqref{eq:consumer-bracket} is bounded below by
\begin{equation}
\label{eq:Q-lower}
 Q_i(y)=4D_i\beta_i^2y^2+2s\beta_i y
 -2D_i\beta_i-e_iv_i.
\end{equation}
For $y\le-L_i$,
\begin{align}
\label{eq:Q-prime}
 Q_i'(y)
 &=8D_i\beta_i^2y+2s\beta_i
 \le-8D_i\beta_i^2L_i+2s\beta_i<0
\end{align}
by \eqref{eq:L-properties}. Therefore $Q_i$ is decreasing as $y$ increases on $(-\infty,-L_i]$, so
\begin{equation}
\label{eq:Q-positive-left}
 Q_i(y)\ge Q_i(-L_i)>0
 \qquad(y\le-L_i).
\end{equation}
It follows from \eqref{eq:consumer-bracket} that the lower consumer inequality is strict on Region III as well.
\end{proof}

\subsection{Completion of the proof and compatibility audit}
\begin{proof}[Proof of Lemma~\ref{lem:ordered-barriers}]
The upper profiles are given by \eqref{eq:upper-U} and \eqref{eq:upper-V}. The lower plant and consumer profiles are given by \eqref{eq:lower-U} and \eqref{eq:lower-V}.

Positivity, boundedness, continuity, and piecewise $C^2$ regularity follow from their explicit formulas. Lemmas~\ref{lem:plant-ordering} and \ref{lem:consumer-ordering} give the componentwise ordering. Lemma~\ref{lem:upper} gives the two upper differential inequalities. Lemmas~\ref{lem:lower-plant} and \ref{lem:lower-consumer} give the two lower differential inequalities.

The full corner set is
\begin{equation}
\label{eq:corner-set}
 \mathcal C:=\{0\}
 \cup\bigcup_{i\in\I}\{a_i,\xi_{V_i}\}.
\end{equation}
At $0$, the upper derivative inequalities are precisely \eqref{eq:upper-corners}. At $a_i$, the lower plant derivative inequality follows from \eqref{eq:lower-U-continuity-corner}. At $\xi_{V_i}$, the lower consumer derivatives agree by \eqref{eq:lower-V-corner}. Other components are smooth at those points. Thus all generalized corner inequalities hold.

The left limits follow immediately from the exponential and Gaussian formulas. At $+\infty$,
\[
 \underline U_i(\xi)=h_i(\xi)\longrightarrow\delta_i^\infty,
 \qquad
 \underline V_i(\xi)\longrightarrow v_i,
\]
where $0<\delta_i^\infty<\sigma_i$ and $0<v_i<1$. The upper profiles converge to $1$. This proves all claims.
\end{proof}

\section{Asymptotic behavior}
\label{sec:asymptotic}

In this section, we determine the asymptotic behavior of the solution
furnished by Lemma~\ref{lem:existence-between-barriers} for the ordered
barriers constructed in Lemma~\ref{lem:ordered-barriers}. The argument
is based on a shrinking-box construction adapted to the
$2N$-component system.

We first recall the following lemma.

\begin{lemma}
\label{lem:fluctuation}
Let $f\in C^2([0,\infty))$ be bounded, and assume that $f'$ and $f''$
are bounded. Set
\[
f^-:=\liminf_{\xi\to+\infty}f(\xi),
\qquad
f^+:=\limsup_{\xi\to+\infty}f(\xi).
\]
Then there exist sequences $\{\xi_n^-\}$ and $\{\xi_n^+\}$ tending to
$+\infty$ such that
\begin{align*}
f(\xi_n^-)&\longrightarrow f^-,
&
f'(\xi_n^-)&\longrightarrow0,
&
\liminf_{n\to\infty}f''(\xi_n^-)&\geq0,
\\
f(\xi_n^+)&\longrightarrow f^+,
&
f'(\xi_n^+)&\longrightarrow0,
&
\limsup_{n\to\infty}f''(\xi_n^+)&\leq0.
\end{align*}
\end{lemma}
\begin{proof}
We prove the assertion associated with $f^+$; the assertion for $f^-$
follows by applying the same argument to $-f$. Choose $t_n\to+\infty$ such that
\[
f(t_n)\to f^+,
\qquad
t_n>2n^2.
\]
Set $\varepsilon_n:=n^{-3}$ and consider
\[
\Phi_n(\xi)
:=
f(\xi)-\varepsilon_n(\xi-t_n)^2
\]
on the interval
\[
I_n:=[t_n-n^2,t_n+n^2].
\]
Since $f$ is bounded and
\[
\varepsilon_n n^4=n,
\]
the maximum of $\Phi_n$ is attained, for all sufficiently large $n$,
at an interior point $\xi_n^+\in I_n$. Therefore
\[
f'(\xi_n^+)
=
2\varepsilon_n(\xi_n^+-t_n),
\qquad
f''(\xi_n^+)\leq2\varepsilon_n.
\]
It follows that
\[
|f'(\xi_n^+)|
\leq2\varepsilon_n n^2=\frac2n\to0,
\qquad
\limsup_{n\to\infty}f''(\xi_n^+)\leq0.
\]
Moreover,
\[
f(\xi_n^+)\geq f(t_n),
\]
while $\xi_n^+\to+\infty$. Hence
\[
f(\xi_n^+)\to f^+.
\]
Applying this argument to $-f$ yields the sequence $\xi_n^-$ and
completes the proof.
\end{proof}

We now establish the shrinking-box result.

\begin{lemma}[Shrinking-box argument]
\label{lem:shrinking-box}
Let $s>0$. Assume that $K=(\kappa_{ij})_{i,j\in\N}$ is nonnegative and
\[
\Theta:=\|K\|_\infty
=
\max_{i\in\N}\sum_{j\in\N}\kappa_{ij}
<1.
\]
Let $(U,V)\in C^2(\mathbb R,\mathbb R^{2N})$ be a bounded strictly positive solution of
\eqref{eq:traveling-wave-system}. Suppose that, for every $i\in\N$,
\begin{equation}
\label{eq:positive-tail-bounds}
0<
\liminf_{\xi\to+\infty}U_i(\xi)
\leq
\limsup_{\xi\to+\infty}U_i(\xi)
\leq1
\end{equation}
and
\begin{equation}
\label{eq:positive-tail-bounds-V}
0<
\liminf_{\xi\to+\infty}V_i(\xi)
\leq
\limsup_{\xi\to+\infty}V_i(\xi)
\leq1.
\end{equation}
Then
\[
\lim_{\xi\to+\infty}(U,V)(\xi)=E^*,
\]
where
\[
E^*=(u^*,v^*),
\qquad
v^*=u^*=(I+K)^{-1}\mathbf 1.
\]
\end{lemma}

\begin{proof}
For each $i\in\N$, write
\[
a_i:=u_i^*=v_i^*.
\]
Since $(I+K)u^*=\mathbf 1$, we have
\begin{equation}
\label{eq:equilibrium-identity}
a_i+\sum_{j\in\N}\kappa_{ij}a_j=1,
\qquad i\in\N.
\end{equation}

We first observe that the first and second derivatives of all components
of $(U,V)$ are bounded on $[0,\infty)$. Indeed, since $(U,V)$ is bounded,
all reaction terms in \eqref{eq:traveling-wave-system} are bounded.
Each component $y$ therefore satisfies an equation of the form
\[
d y''-s y'=h(\xi),
\qquad
\|h\|_{L^\infty}<\infty,
\]
with $d>0$ and $s>0$. Boundedness of $y$ rules out exponential growth
of $y'$, and hence $y'$ is bounded. The equation then implies that
$y''$ is bounded. Consequently, Lemma~\ref{lem:fluctuation} applies to
every component.

For convenience, set
\begin{align*}
U_i^-&:=\liminf_{\xi\to+\infty}U_i(\xi),
&
U_i^+&:=\limsup_{\xi\to+\infty}U_i(\xi),
\\
V_i^-&:=\liminf_{\xi\to+\infty}V_i(\xi),
&
V_i^+&:=\limsup_{\xi\to+\infty}V_i(\xi).
\end{align*}

Choose $\varepsilon>0$ sufficiently small so that
\begin{equation}
\label{eq:epsilon-shrinking}
R\Theta<1,
\qquad
R:=1+\varepsilon.
\end{equation}
For $\theta\in[0,1]$, define
\begin{equation}
\label{eq:shrinking-box}
m_i(\theta):=\theta a_i,
\qquad
M_i(\theta):=\theta a_i+(1-\theta)R.
\end{equation}
Thus
\[
m_i(0)=0,
\qquad
M_i(0)=R,
\qquad
m_i(1)=M_i(1)=a_i.
\]

Consider the set
\[
\mathcal A
:=
\left\{
\theta\in[0,1):
\begin{array}{l}
m_i(\theta)<U_i^-\leq U_i^+<M_i(\theta),\\
m_i(\theta)<V_i^-\leq V_i^+<M_i(\theta),
\quad i\in\N
\end{array}
\right\}.
\]
By \eqref{eq:positive-tail-bounds} and
\eqref{eq:positive-tail-bounds-V}, we have $0\in\mathcal A$.
Moreover, since all inequalities are strict and there are only finitely
many components, $\mathcal A$ contains a nontrivial interval
$[0,\theta_0]$. In particular,
\[
\theta_*:=\sup\mathcal A>0.
\]

We claim that $\theta_*=1$. Suppose, to the contrary, that
$\theta_*<1$. By continuity in $\theta$,
\begin{align}
m_i(\theta_*)&\leq U_i^-\leq U_i^+\leq M_i(\theta_*),
\label{eq:closed-box-U}
\\
m_i(\theta_*)&\leq V_i^-\leq V_i^+\leq M_i(\theta_*).
\label{eq:closed-box-V}
\end{align}
If all inequalities in \eqref{eq:closed-box-U}--\eqref{eq:closed-box-V}
were strict, the box could be shrunk slightly further, contradicting
the definition of $\theta_*$. Hence at least one component touches one
of the boundaries.

We show that none of these contacts is possible.

\medskip

\noindent
\emph{Step 1: no plant component can touch the lower boundary.}

Suppose that, for some $i\in\N$,
\[
U_i^-=m_i(\theta_*).
\]
By Lemma~\ref{lem:fluctuation}, there exists $\xi_n\to+\infty$ such that
\[
U_i(\xi_n)\to m_i(\theta_*),
\qquad
U_i'(\xi_n)\to0,
\qquad
\liminf_{n\to\infty}U_i''(\xi_n)\geq0.
\]
Using \eqref{eq:closed-box-U}--\eqref{eq:closed-box-V}, we obtain
\begin{align*}
\liminf_{n\to\infty}
\left(
1-U_i(\xi_n)-\kappa_{ii}V_i(\xi_n)
-\sum_{\substack{j\in\N,\,j\neq i}}
\kappa_{ij}U_j(\xi_n)
\right)\geq
1-m_i(\theta_*)-\kappa_{ii}M_i(\theta_*)
-\sum_{\substack{j\in\N,\,j\neq i}}
\kappa_{ij}M_j(\theta_*).
\end{align*}
By \eqref{eq:equilibrium-identity},
\begin{align}
1-m_i(\theta)-\kappa_{ii}M_i(\theta)
-\sum_{\substack{j\in\N,\,j\neq i}}
\kappa_{ij}M_j(\theta)=(1-\theta)
\left(
1-R\sum_{j\in\N}\kappa_{ij}
\right)
\geq
(1-\theta)(1-R\Theta)>0
\label{eq:lower-plant-sign}
\end{align}
for every $\theta<1$.

Since $\theta_*>0$, we have $m_i(\theta_*)>0$. Passing to the lower
limit in the $i$th plant equation therefore gives
\[
0
\geq
r_i m_i(\theta_*)
(1-\theta_*)(1-R\Theta)
>0,
\]
which is impossible. Hence
\begin{equation}
\label{eq:no-lower-U-contact}
U_i^->m_i(\theta_*),
\qquad i\in\N.
\end{equation}

\medskip

\noindent
\emph{Step 2: no plant component can touch the upper boundary.}

Suppose that
\[
U_i^+=M_i(\theta_*)
\]
for some $i\in\N$. By Lemma~\ref{lem:fluctuation}, there exists
$\eta_n\to+\infty$ such that
\[
U_i(\eta_n)\to M_i(\theta_*),
\qquad
U_i'(\eta_n)\to0,
\qquad
\limsup_{n\to\infty}U_i''(\eta_n)\leq0.
\]
Moreover,
\begin{align*}
&\limsup_{n\to\infty}
\left(
1-U_i(\eta_n)-\kappa_{ii}V_i(\eta_n)
-\sum_{\substack{j\in\N,\,j\neq i}}
\kappa_{ij}U_j(\eta_n)
\right)\leq
1-M_i(\theta_*)-\kappa_{ii}m_i(\theta_*)
-\sum_{\substack{j\in\N,\,j\neq i}}
\kappa_{ij}m_j(\theta_*).
\end{align*}
Using \eqref{eq:equilibrium-identity}, we find
\begin{align}
&1-M_i(\theta)-\kappa_{ii}m_i(\theta)
-\sum_{\substack{j\in\N,\,j\neq i}}
\kappa_{ij}m_j(\theta)=(1-\theta)(1-R)
=-\varepsilon(1-\theta)<0.
\label{eq:upper-plant-sign}
\end{align}
Passing to the upper limit in the $i$th plant equation gives a strict
negative upper bound for zero, which is a contradiction. Therefore,
\begin{equation}
\label{eq:no-upper-U-contact}
U_i^+<M_i(\theta_*),
\qquad i\in\N.
\end{equation}

\medskip

\noindent
\emph{Step 3: no consumer component can touch the lower boundary.}

Suppose that
\[
V_i^-=m_i(\theta_*)
\]
for some $i\in\N$. By \eqref{eq:no-lower-U-contact},
\[
U_i^->m_i(\theta_*).
\]
Choose a sequence $\xi_n\to+\infty$ such that
\[
V_i(\xi_n)\to m_i(\theta_*),
\qquad
V_i'(\xi_n)\to0,
\qquad
\liminf_{n\to\infty}V_i''(\xi_n)\geq0.
\]
Then
\[
\liminf_{n\to\infty}
\bigl(U_i(\xi_n)-V_i(\xi_n)\bigr)
\geq
U_i^--m_i(\theta_*)>0.
\]
Passing to the lower limit in the $i$th consumer equation yields
\[
0
\geq
e_i m_i(\theta_*)
\bigl(U_i^--m_i(\theta_*)\bigr)
>0,
\]
a contradiction. Thus
\begin{equation}
\label{eq:no-lower-V-contact}
V_i^->m_i(\theta_*),
\qquad i\in\N.
\end{equation}

\medskip

\noindent
\emph{Step 4: no consumer component can touch the upper boundary.}

Suppose that
\[
V_i^+=M_i(\theta_*)
\]
for some $i\in\N$. By \eqref{eq:no-upper-U-contact},
\[
U_i^+<M_i(\theta_*).
\]
Choose $\eta_n\to+\infty$ such that
\[
V_i(\eta_n)\to M_i(\theta_*),
\qquad
V_i'(\eta_n)\to0,
\qquad
\limsup_{n\to\infty}V_i''(\eta_n)\leq0.
\]
Then
\[
\limsup_{n\to\infty}
\bigl(U_i(\eta_n)-V_i(\eta_n)\bigr)
\leq
U_i^+-M_i(\theta_*)<0.
\]
Passing to the upper limit in the $i$th consumer equation yields a
strict negative upper bound for zero, again a contradiction. Therefore,
\begin{equation}
\label{eq:no-upper-V-contact}
V_i^+<M_i(\theta_*),
\qquad i\in\N.
\end{equation}

We have ruled out every possible contact with the boundary of the box.
This contradicts the definition of $\theta_*$ unless
\[
\theta_*=1.
\]
Since the boxes are nested, it follows that, for every $\theta<1$,
\[
m_i(\theta)
<
U_i^-\leq U_i^+
<
M_i(\theta)
\]
and
\[
m_i(\theta)
<
V_i^-\leq V_i^+
<
M_i(\theta).
\]
Letting $\theta\to1^-$ gives
\[
U_i^-=U_i^+=a_i,
\qquad
V_i^-=V_i^+=a_i,
\qquad i\in\N.
\]
Consequently,
\[
\lim_{\xi\to+\infty}U_i(\xi)
=
\lim_{\xi\to+\infty}V_i(\xi)
=
a_i,
\]
which proves
\[
\lim_{\xi\to+\infty}(U,V)(\xi)=E^*.
\]  
\end{proof}

We can now verify the asymptotic conditions for the solution obtained
from the upper--lower solution construction.

\begin{proposition}[Asymptotic behavior of the trapped solution]
\label{prop:asymptotic-behavior}
Let $(U,V)$ be the solution furnished by
Lemma~\ref{lem:existence-between-barriers} for the ordered barriers
constructed in Lemma~\ref{lem:ordered-barriers}. Then
\[
\lim_{\xi\to-\infty}(U,V)(\xi)=(0,0)
\]
and
\[
\lim_{\xi\to+\infty}(U,V)(\xi)=E^*.
\]
In particular, $(U,V)$ satisfies
\eqref{eq:wave-boundary-conditions}.
\end{proposition}

\begin{proof}
By construction of the upper and lower solutions,
\[
\lim_{\xi\to-\infty}
(\underline U,\underline V)(\xi)
=
\lim_{\xi\to-\infty}
(\overline U,\overline V)(\xi)
=
(0,0).
\]
Since
\[
(\underline U,\underline V)
\leq
(U,V)
\leq
(\overline U,\overline V),
\]
the squeeze theorem gives
\[
\lim_{\xi\to-\infty}(U,V)(\xi)=(0,0).
\]

At the opposite end,
\[
\underline U_i(\xi)=h_i(\xi)
\longrightarrow\delta_i^\infty>0,
\qquad
\underline V_i(\xi)\longrightarrow v_i>0
\quad\text{as }\xi\to+\infty.
\]
Consequently,
\[
0<\delta_i^\infty
\leq\liminf_{\xi\to+\infty}U_i(\xi)
\leq\limsup_{\xi\to+\infty}U_i(\xi)
\leq1,
\]
and similarly
\[
0<v_i
\leq\liminf_{\xi\to+\infty}V_i(\xi)
\leq\limsup_{\xi\to+\infty}V_i(\xi)
\leq1.
\]
Lemma~\ref{lem:shrinking-box} therefore implies
\[
\lim_{\xi\to+\infty}(U,V)(\xi)=E^*.
\]
\end{proof}

\begin{proof}[Proof of Theorem~\ref{thm:existence}]
Fix
\[
s\in(s_{\mathrm P},s_{\mathrm C}^{\mathrm{ex}}).
\]
By Lemma~\ref{lem:ordered-barriers}, system~\eqref{eq:TW} admits an
ordered generalized upper--lower pair whose lower components are
strictly positive. Lemma~\ref{lem:existence-between-barriers} then
yields a strictly positive solution
\[
(U,V)\in C^\infty(\mathbb R,\mathbb R^{2N})
\]
of \eqref{eq:TW}, trapped between the corresponding lower and upper
profiles. Proposition~\ref{prop:asymptotic-behavior} gives
\[
\lim_{\xi\to-\infty}(U,V)(\xi)=(0,0),
\qquad
\lim_{\xi\to+\infty}(U,V)(\xi)=E^*.
\]
Hence $(U,V)$ satisfies \eqref{eq:wave-boundary-conditions}.
\end{proof}

\section{Nonexistence for $s<s_{\mathrm{P}}$}
\label{sec 4}

In this section, we prove that the lower critical speed $s_{\mathrm{P}}$ is sharp. More precisely, no strictly positive traveling wave satisfying
\eqref{eq:wave-boundary-conditions} exists when $s<s_{\mathrm{P}}$. The proof is
divided into the cases $s\leq 0$ and $0<s<s_{\mathrm{P}}$.
\begin{proof}[Proof of Theorem~\ref{thm:nonexistence-below-sm}]
Suppose, by contradiction, that $(U,V)\in C^\infty(\mathbb{R},\mathbb{R}^{2N})$ is a strictly positive solution of \eqref{eq:traveling-wave-system}
satisfying
\begin{align*}
    \lim_{\xi\to-\infty}(U,V)(\xi)=(0,0).
\end{align*}
For each $i\in\N$, define
\begin{equation}
\label{eq:Qi-definition}
    \mathcal{Q}_i(\xi)
    :=
    1-U_i(\xi)-\kappa_{ii}V_i(\xi)
    -\sum_{\substack{j\in\N,\,j\neq i}}
      \kappa_{ij}U_j(\xi).
\end{equation}
We have
\begin{equation}
\label{eq:Qi-leading-edge}
    \lim_{\xi\to-\infty}\mathcal{Q}_i(\xi)=1,
    \qquad i\in\N.
\end{equation}

\medskip
\noindent
\textbf{Case 1: $s\leq 0$.} Fix any $i\in\N$. By \eqref{eq:Qi-leading-edge}, there exists
$R\in\mathbb{R}$ such that
\begin{equation}
\label{eq:Qi-half}
\mathcal{Q}_i(\xi)\geq \frac12,
\qquad
\xi\leq R.
\end{equation}
Since
\begin{align*}
\lim_{\xi\to-\infty}U_i(\xi)=0,
\end{align*}
the mean value theorem yields a sequence $a_n\to-\infty$ such that
\begin{align}
\label{eq:derivative-sequence}
U_i'(a_n)\to0.
\end{align}
Indeed, for each sufficiently large $n$, one may choose
$a_n\in(-n-1,-n)$ satisfying
\begin{align*}
U_i'(a_n)
=
U_i(-n)-U_i(-n-1).
\end{align*}

Fix $\xi\leq R$. For all sufficiently large $n$, we have $a_n<\xi$.
Integrating the $i$-th plant equation in
\eqref{eq:traveling-wave-system} over $[a_n,\xi]$ gives
\begin{align}
0
=d_iU_i'(\xi)-d_iU_i'(a_n)
-s\bigl(U_i(\xi)-U_i(a_n)\bigr)+r_i\int_{a_n}^{\xi}
\mathcal{Q}_i(\eta)U_i(\eta)\,d\eta.
\label{eq:integrated-plant-equation}
\end{align}
Since $\mathcal{Q}_i U_i\geq0$ on $(-\infty,R]$, the monotone
convergence theorem, together with \eqref{eq:derivative-sequence},
yields
\begin{equation}
\label{eq:integrated-plant-limit}
d_iU_i'(\xi)-sU_i(\xi)
+r_i\int_{-\infty}^{\xi}
\mathcal{Q}_i(\eta)U_i(\eta)\,d\eta
=0.
\end{equation}
Therefore, by \eqref{eq:Qi-half} and $s\leq0$,
\begin{align*}
d_iU_i'(\xi)
&=
sU_i(\xi)
-r_i\int_{-\infty}^{\xi}
\mathcal{Q}_i(\eta)U_i(\eta)\,d\eta<0,
\qquad
\xi\leq R.
\end{align*}
Thus, $U_i$ is strictly decreasing on $(-\infty,R]$. For any
$\eta<\xi\leq R$, we have $U_i(\xi)<U_i(\eta)$. Letting $\eta\to-\infty$, we obtain $U_i(\xi)\leq0$, which contradicts the strict positivity of $U_i$. Hence no strictly
positive traveling wave exists when $s\leq0$.

\medskip
\noindent
\textbf{Case 2: $0<s<s_{\mathrm{P}}$.} By the definition of $s_{\mathrm{P}}$, there exists $k\in\N$ such that $0<s<2\sqrt{d_kr_k}$, and, hence,
\begin{equation*}
    \gamma_k
    :=
    \frac{r_k}{d_k}-\frac{s^2}{4d_k^2}>0.
\end{equation*}
We now remove the first-order term from the $k$-th plant equation. Set
\begin{equation}
\label{eq:liouville-transform}
    Z_k(\xi)
    :=
    \exp\left(-\frac{s}{2d_k}\xi\right)U_k(\xi).
\end{equation}
Since $U_k>0$, we have
\begin{align*}
    Z_k(\xi)>0,
    \qquad \xi\in\mathbb{R}.
\end{align*}
A direct calculation using the $k$-th plant equation gives
\begin{equation}
\label{eq:transformed-equation}
    Z_k''+A_k(\xi)Z_k=0,\qquad\text{with}\qquad A_k(\xi)
    :=
    \frac{r_k}{d_k}\mathcal{Q}_k(\xi)
    -\frac{s^2}{4d_k^2}.
\end{equation}
By \eqref{eq:Qi-leading-edge}, we obtain
\begin{align*}
    \lim_{\xi\to-\infty}A_k(\xi)
    =
    \frac{r_k}{d_k}-\frac{s^2}{4d_k^2}
    =
    \gamma_k
    >0.
\end{align*}
Consequently, there exists $R\in\mathbb{R}$ such that
\begin{equation}
\label{eq:Ak-positive}
    A_k(\xi)\geq \frac{\gamma_k}{2},
    \qquad \xi\leq R.
\end{equation}
Applying Lemma~\ref{lem:sturm-comparison} to
\eqref{eq:transformed-equation}, with $\delta=\gamma_k/2$, we conclude that $Z_k$ must have a zero on $(-\infty,R]$.
This contradicts the strict positivity of $Z_k$. Therefore, no strictly positive traveling wave exists when
$0<s<s_{\mathrm{P}}$. Combining the two cases completes the proof of Theorem~\ref{thm:nonexistence-below-sm}.
\end{proof}

\begin{remark}
\label{rem:leading-edge-only}
The proof uses only the behavior of the wave profile at the leading
edge:
\begin{align*}
    \lim_{\xi\to-\infty}(U,V)(\xi)=(0,0).
\end{align*}
In particular, neither the weak-interaction condition
$\lVert K\rVert_\infty<1$ nor the convergence to the coexistence
equilibrium at $+\infty$ is required for the nonexistence argument.
\end{remark}

\section{Nonexistence for $s\geq s_{\mathrm C}^{\mathrm{nec}}$}
\label{sec 6}

\begin{lemma}[Exponential decay at the leading edge]
\label{lem:leading-edge-exponential-decay}
Let $s>0$, and let
\begin{align*}
(U,V)=(U_1,\ldots,U_N,V_1,\ldots,V_N)
\end{align*}
be a strictly positive solution of \eqref{eq:traveling-wave-system}
satisfying
\begin{align}
\label{eq:leading-edge-zero-limit}
\lim_{\xi\to-\infty}(U,V)(\xi)=(0,0).
\end{align}
Then there exist constants $C>0$, $\eta>0$, and $R<0$ such that
\begin{align}
\label{eq:full-leading-edge-exponential-bound}
\sum_{i\in\N}
\left(
 |U_i(\xi)|+|U_i'(\xi)|+|V_i(\xi)|+|V_i'(\xi)|
\right)
\leq C e^{\eta\xi},
\qquad \xi\leq R.
\end{align}
In particular, for every $i\in\N$,
\begin{align}
\label{eq:Ui-minus-Vi-L1}
U_i-V_i\in L^1(( -\infty,0]).
\end{align}
\end{lemma}

\begin{proof}
We divide the proof into five steps.

\medskip
\noindent
\textit{Step 1: First-order formulation and spectral splitting.}
Set
\begin{align*}
P_i:=U_i',
\qquad
Q_i:=V_i',
\qquad i\in\N,
\end{align*}
and write
\begin{align*}
\mathbf X:=(U,P,V,Q)\in\mathbb R^{4N}.
\end{align*}
The traveling-wave system becomes
\begin{equation}
\label{eq:first-order-traveling-wave-system}
\left\{
\begin{aligned}
U_i'&=P_i,\\
P_i'
&=
\frac{s}{d_i}P_i
-\frac{r_i}{d_i}U_i
+\frac{r_i}{d_i}U_i
\left(
 U_i+\kappa_{ii}V_i
 +\sum_{\substack{j\in\N,\,j\neq i}}
  \kappa_{ij}U_j
\right),\\
V_i'&=Q_i,\\
Q_i'
&=
\frac{s}{D_i}Q_i
-\frac{e_i}{D_i}U_iV_i
+\frac{e_i}{D_i}V_i^2,
\end{aligned}
\right.
\qquad i\in\N.
\end{equation}
Thus,
\begin{align}
\label{eq:X-linear-plus-quadratic}
\mathbf X'=L\mathbf X+\mathcal F(\mathbf X),
\qquad
\mathcal F(\mathbf X)=O(\|\mathbf X\|^2)
\quad\text{as}\quad\mathbf X\to0.
\end{align}
The plant block of $L$ corresponding to $(U_i,P_i)$ is
\begin{align*}
L_i^{\rm p}
=
\begin{pmatrix}
0&1\\
-r_i/d_i&s/d_i
\end{pmatrix},
\end{align*}
whose eigenvalues are
\begin{align}
\label{eq:plant-linear-eigenvalues}
\lambda_i^{\pm}
=
\frac{s\pm\sqrt{s^2-4d_ir_i}}{2d_i}.
\end{align}
In either the real or the nonreal case,
\begin{align*}
\operatorname{Re}\lambda_i^{\pm}>0
\qquad\text{because }s>0.
\end{align*}
The consumer block corresponding to $(V_i,Q_i)$ is
\begin{align*}
L_i^{\rm c}
=
\begin{pmatrix}
0&1\\
0&s/D_i
\end{pmatrix},
\end{align*}
with eigenvalues
\begin{align*}
0,
\qquad
\frac{s}{D_i}>0.
\end{align*}
Consequently, the center eigenspace of $L$ is
\begin{align}
\label{eq:center-eigenspace-original-system}
E^{\rm c}
=
\left\{
(U,P,V,Q)\in\mathbb R^{4N}:
U=0,\ P=0,\ Q=0
\right\}.
\end{align}
All remaining eigenvalues of $L$ have strictly positive real parts.

\medskip
\noindent
\textit{Step 2: Construction of a convenient center manifold.}
Consider the subspace
\begin{align}
\label{eq:invariant-subspace-S}
\mathfrak S
:=
\left\{
(U,P,V,Q)\in\mathbb R^{4N}:
U=0,\ P=0
\right\}.
\end{align}
It is invariant under \eqref{eq:first-order-traveling-wave-system}.
Indeed, if $U=P=0$, then the first two equations in
\eqref{eq:first-order-traveling-wave-system} give $U'=P'=0$.
The restriction of \eqref{eq:first-order-traveling-wave-system} to
$\mathfrak S$ is
\begin{equation}
\label{eq:restricted-consumer-system}
\left\{
\begin{aligned}
V_i'&=Q_i,\\
Q_i'&=\frac{s}{D_i}Q_i+\frac{e_i}{D_i}V_i^2,
\end{aligned}
\right.
\qquad i\in\N.
\end{equation}
For each $i\in\N$, the center-manifold theorem applied to the
corresponding two-dimensional subsystem yields a local center manifold
of the form
\begin{align*}
Q_i=\mathfrak{h}_i(V_i),
\qquad
\mathfrak{h}_i(0)=\mathfrak{h}_i'(0)=0.
\end{align*}
Its invariance equation is
\begin{align}
\label{eq:hi-invariance-equation}
\mathfrak{h}_i'(v)\mathfrak{h}_i(v)
=
\frac{s}{D_i}\mathfrak{h}_i(v)
+
\frac{e_i}{D_i}v^2.
\end{align}
Writing
\begin{align*}
\mathfrak{h}_i(v)=\chi_iv^2+O(v^3)
\qquad\text{as}\qquad v\to0
\end{align*}
and comparing the quadratic terms in
\eqref{eq:hi-invariance-equation}, we obtain
\begin{align*}
0
=
\frac{s}{D_i}\chi_i+\frac{e_i}{D_i}.
\end{align*}
Therefore,
\begin{align}
\label{eq:hi-quadratic-expansion}
\mathfrak{h}_i(v)
=
-\frac{e_i}{s}v^2+O(v^3)
\qquad\text{as }v\to0.
\end{align}
Set
\begin{align*}
\mathfrak{h}(V):=(\mathfrak{h}_1(V_1),\ldots,\mathfrak{h}_N(V_N))
\end{align*}
and define
\begin{align}
\label{eq:chosen-center-manifold}
W^{\rm c}
:=
\left\{
(U,P,V,Q):
U=0,\ P=0,\ Q=\mathfrak{h}(V)
\right\}.
\end{align}
This is a local invariant manifold for the full system because
$\mathfrak S$ is invariant.  Moreover,
\begin{align*}
T_0W^{\rm c}
=
\{U=0,\ P=0,\ Q=0\}
=
E^{\rm c}.
\end{align*}
Thus, $W^{\rm c}$ is a local center manifold of the full system.
Notice that $\mathfrak S$ itself is not a center manifold: it has
dimension $2N$, whereas $\dim E^{\rm c}=N$.

\medskip
\noindent
\textit{Step 3: Exponential tracking of an orbit on the chosen local
center manifold.}
Introduce the reversed independent variable
\begin{align*}
\tau:=-\xi,
\qquad
\widetilde{\mathbf X}(\tau):=\mathbf X(-\tau).
\end{align*}
Then
\begin{align}
\label{eq:reversed-first-order-system}
\dot{\widetilde{\mathbf X}}
=
-L\widetilde{\mathbf X}
-
\mathcal F(\widetilde{\mathbf X}).
\end{align}
The spectrum of $-L$ consists of $N$ zero eigenvalues and $3N$
eigenvalues with strictly negative real parts.  We next write down an
explicit linear change of coordinates which separates these stable and
center directions.

Set
\begin{align*}
\mathcal D
:=
\operatorname{diag}\left(\frac{D_1}{s},\ldots,\frac{D_N}{s}\right)
\end{align*}
and define
\begin{align}
\label{eq:explicit-stable-center-coordinate-change}
x
&:=
(\widetilde U,\widetilde P,\widetilde Q)
\in\mathbb R^{3N},
\\
y
&:=
\widetilde V-\mathcal D\widetilde Q
\in\mathbb R^N.
\end{align}
This is a fixed invertible linear change of coordinates, since
\begin{align*}
\widetilde V=y+\mathcal D\widetilde Q.
\end{align*}
In these coordinates, \eqref{eq:reversed-first-order-system} takes the
stable--center form
\begin{equation}
\label{eq:KA-autonomous-form}
\left\{
\begin{aligned}
\dot x&=A^-x+p(x,y),\\
\dot y&=q(x,y),
\end{aligned}
\right.
\end{equation}
where
\begin{align}
\label{eq:stable-linear-block}
A^-
=
\operatorname{diag}
\left(
A_1^{\mathrm{p},-},\ldots,A_N^{\mathrm{p},-},
-\frac{s}{D_1},\ldots,-\frac{s}{D_N}
\right),
\qquad
A_i^{\mathrm{p},-}
:=
\begin{pmatrix}
0&-1\\
r_i/d_i&-s/d_i
\end{pmatrix}.
\end{align}
Indeed, the eigenvalues of $A_i^{\mathrm{p},-}$ are
$-\lambda_i^\pm$, where $\lambda_i^\pm$ are defined in
\eqref{eq:plant-linear-eigenvalues}.  Hence
\begin{align*}
\operatorname{Re}\sigma(A^-)<0.
\end{align*}
Moreover, the nonlinearities satisfy
\begin{align}
\label{eq:pq-quadratic-order}
p(0,0)=q(0,0)=0,
\qquad
Dp(0,0)=Dq(0,0)=0.
\end{align}
For example, the $i$-th center equation is
\begin{align*}
\dot y_i=
\dot{\widetilde V}_i-\frac{D_i}{s}\dot{\widetilde Q}_i=
-\frac{e_i}{s}\widetilde V_i
\bigl(\widetilde U_i-\widetilde V_i\bigr),
\end{align*}
which has no linear term.  The remaining nonlinear terms are also at
least quadratic because
$\mathcal F(\mathbf X)=O(\lVert\mathbf X\rVert^2)$.

We now verify that the specific local center manifold constructed in
Step~2 is an admissible chosen center manifold for
\eqref{eq:KA-autonomous-form}.  Reversing time does not change
invariant sets.  On $W^{\rm c}$, we have
\begin{align*}
U=P=0,
\qquad
Q=\mathfrak{h}(V).
\end{align*}
Consequently, in the coordinates \eqref{eq:explicit-stable-center-coordinate-change},
$W^{\rm c}$ is parametrized by
\begin{align*}
x
&=
(0,0,\mathfrak{h}(V)),
\\
y
&=
\Psi(V)
:=
V-\mathcal D \mathfrak{h}(V).
\end{align*}
Since $\mathfrak{h}(0)=0$ and $D\mathfrak{h}(0)=0$, we have
\begin{align*}
\Psi(0)=0,
\qquad
D\Psi(0)=I.
\end{align*}
The inverse function theorem therefore yields neighborhoods
$\mathcal V,\mathcal Y\subset\mathbb R^N$ of the origin such that
$\Psi:\mathcal V\to\mathcal Y$ is a $C^1$ diffeomorphism.  Denote its
local inverse by $\psi$.  In the stable--center coordinates, the same
chosen manifold is thus the graph
\begin{align}
\label{eq:chosen-center-manifold-in-KA-coordinates}
x=H(y),
\qquad
H(y)
:=
\bigl(0,0,\mathfrak{h}(\psi(y))\bigr),
\qquad
y\in\mathcal Y.
\end{align}
In particular,
\begin{align*}
H(0)=0,
\qquad
DH(0)=0.
\end{align*}

We apply the autonomous asymptotic-phase theorem of Knobloch and
Aulbach \cite[Theorem~3]{knobloch1982role} to
\eqref{eq:KA-autonomous-form} with the chosen local center manifold
\eqref{eq:chosen-center-manifold-in-KA-coordinates}.  In the notation of
that theorem, $x$ is the stable variable, $y$ is the center variable,
and the reduced equation is
\begin{align}
\label{eq:KA-reduced-equation}
\dot{\overline y}
=
q\bigl(H(\overline y),\overline y\bigr).
\end{align}
The theorem states that, for every sufficiently small initial state,
there exists an initial center phase for
\eqref{eq:KA-reduced-equation} such that the corresponding orbit on the
chosen center manifold tracks the full orbit with an exponentially decaying error.  For a precise formulation and proof of the underlying
reduction principle, see also \cite{Aulbach1982Reduction}.

We use this result in its standard local sense.  Equivalently, one may
first multiply $p$ and $q$ by a smooth cutoff which is identically one
in a sufficiently small neighborhood of the origin.  Since
\eqref{eq:leading-edge-zero-limit} gives
\begin{align*}
\widetilde{\mathbf X}(\tau)\longrightarrow0
\qquad
\text{as }\tau\to+\infty,
\end{align*}
we may choose $\tau_0>0$ so large that the shifted orbit
$\widetilde{\mathbf X}(\tau_0+\cdot)$ lies in this neighborhood.  The
tracking estimate below implies that the tracked center orbit also
tends to zero; after increasing $\tau_0$ if necessary, both trajectories
remain in the region where the cutoff is identically one.  Hence the
cutoff does not alter either trajectory on the tail under consideration.

To spell out the application of the theorem, let
\begin{align*}
(\widehat x(t),\widehat y(t))
:=
\mathcal T\widetilde{\mathbf X}(\tau_0+t),
\qquad t\geq0,
\end{align*}
where $\mathcal T$ denotes the linear transformation in
\eqref{eq:explicit-stable-center-coordinate-change}.  By
\cite[Theorem~3]{knobloch1982role}, there exists a solution
$\overline y(t)$ of \eqref{eq:KA-reduced-equation} and constants
$C_1>0$, $\eta>0$ such that
\begin{align}
\label{eq:KA-tracking-estimate-in-split-coordinates}
\lVert\widehat x(t)-H(\overline y(t))\rVert
+
\lVert\widehat y(t)-\overline y(t)\rVert
\leq
C_1e^{-\eta t},
\qquad t\geq0.
\end{align}
Define the corresponding trajectory on the chosen center manifold by
\begin{align*}
\widetilde{\mathbf Z}(\tau_0+t)
:=
\mathcal T^{-1}
\bigl(H(\overline y(t)),\overline y(t)\bigr),
\qquad t\geq0.
\end{align*}
Since $\mathcal T^{-1}$ is a fixed linear map,
\eqref{eq:KA-tracking-estimate-in-split-coordinates} yields, after
changing the constant,
\begin{align}
\label{eq:tracking-estimate-reversed-time}
\lVert\widetilde{\mathbf X}(\tau)
-
\widetilde{\mathbf Z}(\tau)\rVert
\leq
C_0e^{-\eta(\tau-\tau_0)},
\qquad
\tau\geq\tau_0.
\end{align}
Let
\begin{align*}
R:=-\tau_0,
\qquad
\mathbf Z(\xi):=\widetilde{\mathbf Z}(-\xi).
\end{align*}
Then \eqref{eq:tracking-estimate-reversed-time} becomes
\begin{align}
\label{eq:tracking-estimate-original-time}
\lVert\mathbf X(\xi)-\mathbf Z(\xi)\rVert
\leq
C_0e^{\eta(\xi-R)},
\qquad
\xi\leq R.
\end{align}
Since $\mathbf X(\xi)\to0$ as $\xi\to-\infty$, it follows from
\eqref{eq:tracking-estimate-original-time} that
\begin{align}
\label{eq:tracked-center-orbit-zero-limit}
\mathbf Z(\xi)\longrightarrow0
\qquad
\text{as }\xi\to-\infty.
\end{align}

\medskip
\noindent
\textit{Step 4: Positivity excludes every nonzero center orbit.}
Because $\mathbf Z$ lies on the chosen manifold $W^{\rm c}$, it has the
form
\begin{align*}
\mathbf Z(\xi)
=
(0,0,z(\xi),\mathfrak{h}(z(\xi))),
\qquad
\xi\leq R.
\end{align*}
Moreover, $\mathbf Z$ is a trajectory of the original $\xi$-system.
Therefore,
\begin{align}
\label{eq:center-reduced-equation}
z_i'
=
\mathfrak{h}_i(z_i)
=
-\frac{e_i}{s}z_i^2+O(z_i^3),
\qquad
i\in\N.
\end{align}
By \eqref{eq:tracked-center-orbit-zero-limit},
\begin{align*}
z_i(\xi)\longrightarrow0
\qquad
\text{as }\xi\to-\infty.
\end{align*}

Suppose, by contradiction, that $z_i\not\equiv0$ for some $i\in\N$.
By uniqueness for the scalar equation $z_i'=\mathfrak{h}_i(z_i)$, the function
$z_i$ cannot vanish at any point of $(-\infty,R]$.  Hence
$1/z_i$ is well-defined on this interval.  From
\eqref{eq:center-reduced-equation},
\begin{align}
\label{eq:reciprocal-center-equation}
\left(\frac{1}{z_i}\right)'
=
-\frac{z_i'}{z_i^2}
=
\frac{e_i}{s}+O(z_i)
\longrightarrow
\frac{e_i}{s}
\qquad
\text{as }\xi\to-\infty.
\end{align}
For completeness, fix $R_1\leq R$ and write
\begin{align*}
g_i(\xi)
:=
\left(\frac{1}{z_i(\xi)}\right)'
-
\frac{e_i}{s}.
\end{align*}
Then $g_i(\xi)\to0$ as $\xi\to-\infty$, and for every
$\xi\leq R_1$,
\begin{align*}
\frac{1}{z_i(\xi)}=
\frac{1}{z_i(R_1)}
-
\int_{\xi}^{R_1}
\left(\frac{e_i}{s}+g_i(t)\right)\,dt=
\frac{e_i}{s}\xi
+
\left(
\frac{1}{z_i(R_1)}-
\frac{e_i}{s}R_1
\right)
-
\int_{\xi}^{R_1}g_i(t)\,dt.
\end{align*}
Since $g_i(t)\to0$ as $t\to-\infty$, the last integral is
$o(|\xi|)$ as $\xi\to-\infty$.  Consequently,
\begin{align}
\label{eq:center-orbit-algebraic-asymptotics}
\frac{1}{z_i(\xi)}
&=
\frac{e_i}{s}\xi+o(|\xi|),
\\
z_i(\xi)
&=
\frac{s}{e_i\xi}(1+o(1))
\qquad
\text{as}\qquad \xi\to-\infty.
\end{align}
In particular,
\begin{align*}
z_i(\xi)<0
\qquad
\text{for all sufficiently negative }\xi.
\end{align*}
On the other hand, \eqref{eq:tracking-estimate-original-time} implies
componentwise that
\begin{align*}
V_i(\xi)-z_i(\xi)
=
O(e^{\eta\xi})
=
o(|\xi|^{-1})
\qquad
\text{as }\xi\to-\infty.
\end{align*}
Combining this estimate with
\eqref{eq:center-orbit-algebraic-asymptotics}, we obtain
\begin{align*}
V_i(\xi)
=
\frac{s}{e_i\xi}(1+o(1))<0
\end{align*}
for all sufficiently negative $\xi$.  This contradicts the strict
positivity of $V_i$.  Therefore,
\begin{align*}
z_i\equiv0
\qquad
\text{for every }i\in\N.
\end{align*}
Since $\mathfrak{h}(0)=0$, we conclude that
\begin{align}
\label{eq:center-orbit-zero}
\mathbf Z\equiv0.
\end{align}

\medskip
\noindent
\textit{Step 5: Exponential decay and integrability.}
Combining \eqref{eq:tracking-estimate-original-time} and
\eqref{eq:center-orbit-zero}, we obtain
\begin{align*}
\|\mathbf X(\xi)\|
\leq
C_0e^{\eta(\xi-R)},
\qquad
\xi\leq R.
\end{align*}
After enlarging the constant, this proves
\eqref{eq:full-leading-edge-exponential-bound}.
In particular,
\begin{align*}
|U_i(\xi)-V_i(\xi)|
\leq
C e^{\eta\xi},
\qquad
\xi\leq R.
\end{align*}
Since $U_i-V_i$ is continuous on the compact interval $[R,0]$, we
obtain
\begin{align*}
U_i-V_i\in L^1(( -\infty,0]),
\end{align*}
which is \eqref{eq:Ui-minus-Vi-L1}.
\end{proof}

\begin{lemma}
\label{lem:degenerate-leading-edge}
Let $\beta>0$ and assume that
\begin{align*}
a\in C((-\infty,0])\cap L^1((-\infty,0]).
\end{align*}
Suppose that $V\not\equiv0$ is a solution of
\begin{equation}\label{eq:degenerate-linear-equation}
V''-\beta V'+a(\xi)V=0,
\qquad
\xi\in(-\infty,0],
\end{equation}
satisfying
\begin{align*}
\lim_{\xi\to-\infty}V(\xi)=0.
\end{align*}
Then there exists a constant $A\neq0$ such that
\begin{align*}
V(\xi)
=
Ae^{\beta\xi}(1+o(1)),
\qquad
V'(\xi)
=
\beta Ae^{\beta\xi}(1+o(1)),
\qquad
\xi\to-\infty.
\end{align*}
Consequently,
\begin{align*}
\lim_{\xi\to-\infty}
\frac{V'(\xi)}{V(\xi)}
=
\beta.
\end{align*}
If $V(\xi)>0$ for every $\xi\leq0$, then $A>0$.
\end{lemma}

\begin{proof}
Choose $R_0<0$ sufficiently negative so that
\begin{equation}\label{eq:small-tail-a}
\alpha_{R_0}
:=
\frac{1}{\beta}
\int_{-\infty}^{R_0}|a(t)|\,dt
<
\frac12.
\end{equation}
Let
\begin{align*}
X:=C_b((-\infty,R_0])
\end{align*}
be the Banach space of bounded continuous functions on
$(-\infty,R_0]$, equipped with the supremum norm.  Define
$\mathcal T:X\to X$ by
\begin{equation}\label{eq:fixed-point-operator-leading-edge}
(\mathcal Ty)(\xi)
:=
1-\frac{1}{\beta}
\int_{-\infty}^{\xi}
\left(
1-e^{-\beta(\xi-t)}
\right)
a(t)y(t)\,dt.
\end{equation}
The map $\mathcal T$ is well-defined.  Indeed, for every $y\in X$,
\begin{align*}
\begin{aligned}
|(\mathcal Ty)(\xi)|
\leq
1+
\frac{1}{\beta}
\int_{-\infty}^{\xi}
|a(t)|\,|y(t)|\,dt\leq
1+\alpha_{R_0}\|y\|_\infty,
\qquad
\xi\leq R_0.
\end{aligned}
\end{align*}
Hence $\mathcal Ty$ is bounded.  Its continuity follows from
\eqref{eq:fixed-point-operator-leading-edge} and the dominated
convergence theorem.

For any $y_1,y_2\in X$, we have
\begin{align*}
\begin{aligned}
|(\mathcal Ty_1)(\xi)-(\mathcal Ty_2)(\xi)|
\leq
\frac{1}{\beta}
\int_{-\infty}^{\xi}
|a(t)|\,|y_1(t)-y_2(t)|\,dt\leq
\alpha_{R_0}
\|y_1-y_2\|_\infty.
\end{aligned}
\end{align*}
Taking the supremum over $\xi\leq R_0$, we obtain
\begin{align*}
\|\mathcal Ty_1-\mathcal Ty_2\|_\infty
\leq
\alpha_{R_0}
\|y_1-y_2\|_\infty.
\end{align*}
Since $\alpha_{R_0}<1$, the map $\mathcal T$ is a contraction.
Therefore, by the Banach fixed-point theorem, there exists a unique
$y\in X$ such that
\begin{equation}\label{eq:y-fixed-point}
y=\mathcal Ty.
\end{equation}

Using \eqref{eq:y-fixed-point}, we estimate
\begin{align*}
|y(\xi)-1|
\leq
\frac{\|y\|_\infty}{\beta}
\int_{-\infty}^{\xi}|a(t)|\,dt.
\end{align*}
Since $a\in L^1((-\infty,0])$, it follows that
\begin{equation}\label{eq:y-limit}
\lim_{\xi\to-\infty}y(\xi)=1.
\end{equation}
Differentiating \eqref{eq:y-fixed-point}, we obtain
\begin{equation}\label{eq:y-prime-formula}
y'(\xi)
=
-
\int_{-\infty}^{\xi}
e^{-\beta(\xi-t)}
a(t)y(t)\,dt.
\end{equation}
Consequently,
\begin{align*}
|y'(\xi)|
\leq
\|y\|_\infty
\int_{-\infty}^{\xi}|a(t)|\,dt,
\end{align*}
and hence
\begin{equation}\label{eq:y-prime-limit}
\lim_{\xi\to-\infty}y'(\xi)=0.
\end{equation}
Differentiating \eqref{eq:y-prime-formula} once more yields
\begin{align*}
y''(\xi)
=
-a(\xi)y(\xi)-\beta y'(\xi).
\end{align*}
Therefore,
\begin{equation}\label{eq:y-equation}
y''+\beta y'+a(\xi)y=0.
\end{equation}

Define
\begin{align*}
\phi(\xi):=e^{\beta\xi}y(\xi).
\end{align*}
Using \eqref{eq:y-equation}, we calculate
\begin{align*}
\begin{aligned}
\phi''-\beta\phi'+a(\xi)\phi=
e^{\beta\xi}
\left(
y''+\beta y'+a(\xi)y
\right)=0.
\end{aligned}
\end{align*}
Thus $\phi$ solves \eqref{eq:degenerate-linear-equation}.  Moreover,
by \eqref{eq:y-limit} and \eqref{eq:y-prime-limit},
\begin{equation}\label{eq:phi-asymptotics}
\phi(\xi)
=
e^{\beta\xi}(1+o(1)),
\qquad
\phi'(\xi)
=
\beta e^{\beta\xi}(1+o(1)),
\qquad
\xi\to-\infty.
\end{equation}

Next, let us construct a second linearly independent solution. By \eqref{eq:y-limit}, after decreasing $R_0$ if necessary, we may
assume that
\begin{align*}
y(\xi)\geq\frac12,
\qquad
\xi\leq R_0.
\end{align*}
In particular,
\begin{align*}
\phi(\xi)\neq0,
\qquad
\xi\leq R_0.
\end{align*}
Define
\begin{equation}\label{eq:psi-definition}
\psi(\xi)
:=
\beta\phi(\xi)I(\xi) \qquad\text{with}\qquad I(\xi)
:=
\int_{\xi}^{R_0}
\frac{e^{\beta t}}{\phi(t)^2}\,dt.
\end{equation}
We now verify directly that $\psi$ is also a solution of
\eqref{eq:degenerate-linear-equation}.  Since
\begin{align*}
I'(\xi)
=
-\frac{e^{\beta\xi}}{\phi(\xi)^2} \qquad\text{and}\qquad I''(\xi)
=
\left(
2\frac{\phi'(\xi)}{\phi(\xi)}
-\beta
\right)
\frac{e^{\beta\xi}}{\phi(\xi)^2},
\end{align*}
we have
\begin{equation}\label{eq:I-identity}
\phi I''
+
(2\phi'-\beta\phi)I'
=
0.
\end{equation}
Let
\begin{align*}
\mathcal L w
:=
w''-\beta w'+a(\xi)w.
\end{align*}
Since $\mathcal L\phi=0$, we calculate
\begin{align*}
\begin{aligned}
\mathcal L\psi=
\beta\mathcal L(\phi I)=
\beta
\left[
I\mathcal L\phi
+
\phi I''
+
(2\phi'-\beta\phi)I'
\right]=0
\end{aligned}
\end{align*}
by \eqref{eq:I-identity}.  Hence $\psi$ is a second solution.

To verify linear independence, we compute the Wronskian:
\begin{align*}
\begin{aligned}
W(\phi,\psi):=
\phi\psi'-\phi'\psi=
\beta\phi^2 I'=
-\beta e^{\beta\xi}.
\end{aligned}
\end{align*}
Since
\begin{align*}
W(\phi,\psi)(\xi)\neq0
\qquad
\text{for every }\xi\leq R_0,
\end{align*}
the functions $\phi$ and $\psi$ are linearly independent. It remains to determine the asymptotic behavior of $\psi$.  From
\eqref{eq:phi-asymptotics},
\begin{align*}
\frac{e^{\beta\xi}}{\phi(\xi)^2}
=
e^{-\beta\xi}(1+o(1)).
\end{align*}
Thus
\begin{align*}
I(\xi)
=
\int_{\xi}^{R_0}
\frac{e^{\beta t}}{\phi(t)^2}\,dt
\longrightarrow+\infty
\qquad\text{as}\qquad \xi\to-\infty.
\end{align*}
Applying l'Hopital's rule, we obtain
\begin{align*}
\begin{aligned}
\lim_{\xi\to-\infty}
\frac{I(\xi)}{\beta^{-1}e^{-\beta\xi}}=
\lim_{\xi\to-\infty}
\frac{-e^{\beta\xi}/\phi(\xi)^2}
{-e^{-\beta\xi}}=
\lim_{\xi\to-\infty}
\frac{e^{2\beta\xi}}{\phi(\xi)^2}=1.
\end{aligned}
\end{align*}
Therefore,
\begin{equation}\label{eq:psi-asymptotics}
\psi(\xi)
=
\beta\phi(\xi)I(\xi)
=
1+o(1),
\qquad
\xi\to-\infty.
\end{equation}

Finally, since $\phi$ and $\psi$ are linearly independent solutions of the
second-order equation \eqref{eq:degenerate-linear-equation}, they form
a fundamental system on $(-\infty,R_0]$.  Hence there exist constants
$A,B\in\mathbb R$ such that
\begin{align*}
V=A\phi+B\psi.
\end{align*}
Using
\begin{align*}
V(\xi)\to0,
\qquad
\phi(\xi)\to0,
\qquad
\psi(\xi)\to1
\qquad\text{as}\qquad\xi\to-\infty,
\end{align*}
we conclude that $B=0$. Therefore, $V=A\phi$. Since $V\not\equiv0$, we have $A\neq0$.  By
\eqref{eq:phi-asymptotics},
\begin{align*}
V(\xi)
=
Ae^{\beta\xi}(1+o(1)),
\qquad
V'(\xi)
=
\beta Ae^{\beta\xi}(1+o(1)),
\qquad
\xi\to-\infty.
\end{align*}
It follows that
\begin{align*}
\lim_{\xi\to-\infty}
\frac{V'(\xi)}{V(\xi)}
=
\beta.
\end{align*}
Finally, if $V$ is positive, then $\phi$ is positive for sufficiently
negative $\xi$, and hence $A>0$.
\end{proof}

\begin{proof}[Proof of Theorem~\ref{thm:upper-speed-nonexistence}]
We first observe that
\begin{equation}\label{eq:Ui-less-than-one}
0<U_i(\xi)<1,
\qquad
\xi\in\mathbb R,\quad i\in\N.
\end{equation}
Indeed, positivity follows from the assumptions.  Suppose that
$U_i(\xi_0)\geq1$ at a global maximum point $\xi_0$.  Since
\begin{align*}
U_i'(\xi_0)=0,
\qquad
U_i''(\xi_0)\leq0,
\end{align*}
the $i$-th plant equation gives
\begin{align*}
\begin{aligned}
0=
d_iU_i''(\xi_0)
-sU_i'(\xi_0)+
r_iU_i(\xi_0)
\left(
1-U_i(\xi_0)
-\kappa_{ii}V_i(\xi_0)
-\sum_{\substack{j\in\N,\,j\neq i}}
 \kappa_{ij}U_j(\xi_0)
\right)<0,
\end{aligned}
\end{align*}
which is a contradiction.  Thus \eqref{eq:Ui-less-than-one} holds.

Fix $i\in\N$ and define
\begin{align*}
q_i(\xi):=\frac{V_i'(\xi)}{V_i(\xi)}.
\end{align*}
This function is well-defined and continuously differentiable because
$V_i(\xi)>0$ for every $\xi\in\mathbb R$.  Dividing the $i$-th
consumer equation by $D_iV_i$, we obtain
\begin{align*}
\frac{V_i''}{V_i}
-\frac{s}{D_i}\frac{V_i'}{V_i}
+
\frac{e_i}{D_i}(U_i-V_i)
=0.
\end{align*}
Since
\begin{align*}
\frac{V_i''}{V_i}
=
q_i'+q_i^2,
\end{align*}
it follows that
\begin{equation}\label{eq:consumer-riccati}
q_i'
=
\frac{s}{D_i}q_i
-q_i^2
-\frac{e_i}{D_i}(U_i-V_i).
\end{equation}

We next determine the limits of $q_i$ at the two ends of the real
line. By Lemma~\ref{lem:leading-edge-exponential-decay}, we have
\begin{align*}
a_i(\xi)
:=
\frac{e_i}{D_i}\bigl(U_i(\xi)-V_i(\xi)\bigr)\in L^1((-\infty,0]).
\end{align*}
The consumer equation can be written as
\begin{align*}
V_i''
-\frac{s}{D_i}V_i'
+
a_i(\xi)V_i
=0.
\end{align*}
Since $V_i\not\equiv0$ and $V_i(\xi)\to0$ as $\xi\to-\infty$,
Lemma~\ref{lem:degenerate-leading-edge}, applied with $\beta=s/D_i$, yields
\begin{align*}
V_i(\xi)
=
A_i e^{(s/D_i)\xi}(1+o(1)),
\end{align*}
and
\begin{align*}
V_i'(\xi)
=
\frac{s}{D_i}
A_i e^{(s/D_i)\xi}(1+o(1)),
\qquad
\xi\to-\infty,
\end{align*}
for some constant $A_i>0$.  Therefore,
\begin{equation}\label{eq:q-minus-infinity}
\lim_{\xi\to-\infty}q_i(\xi)
=
\frac{s}{D_i}.
\end{equation}

Since \eqref{eq:wave-boundary-conditions} holds, we have
\begin{align*}
\tilde{a}_i(\xi)
:=
\frac{e_i}{D_i}
V_i(\xi)\bigl(U_i(\xi)-V_i(\xi)\bigr)
\longrightarrow0\qquad\text{as}\qquad \xi\to+\infty.
\end{align*}
Let
\begin{align*}
W_i(\xi):=V_i'(\xi),
\qquad
\beta_i:=\frac{s}{D_i}>0.
\end{align*}
The consumer equation becomes
\begin{equation}\label{eq:Wi-linear-equation}
W_i'-\beta_iW_i=-\tilde{a}_i.
\end{equation}
Multiplying by $e^{-\beta_i\xi}$ gives
\begin{align*}
\left(
e^{-\beta_i\xi}W_i(\xi)
\right)'
=
-e^{-\beta_i\xi} \tilde{a}_i(\xi).
\end{align*}
Since $\tilde{a}_i$ is bounded for sufficiently large $\xi$,
the function on the right-hand side is integrable on every interval
$[R,+\infty)$.  Hence the limit
\begin{align*}
L_i
:=
\lim_{\xi\to+\infty}
e^{-\beta_i\xi}W_i(\xi)
\end{align*}
exists. We claim that $L_i=0$.  Otherwise,
\begin{align*}
W_i(\xi)
=
e^{\beta_i\xi}(L_i+o(1)),
\qquad
\xi\to+\infty.
\end{align*}
For sufficiently large $\xi$, the derivative $W_i=V_i'$ would then
have a fixed sign and satisfy
\begin{align*}
|V_i'(\xi)|
\geq
\frac{|L_i|}{2}e^{\beta_i\xi}.
\end{align*}
Integrating this inequality would imply that $V_i(\xi)$ is unbounded
as $\xi\to+\infty$, contradicting
\begin{align*}
V_i(\xi)\longrightarrow u_i^*.
\end{align*}
Thus $L_i=0$.

Integrating \eqref{eq:Wi-linear-equation} from $\xi$ to $+\infty$, we
obtain
\begin{align*}
V_i'(\xi)
=
\int_{\xi}^{+\infty}
e^{-\beta_i(t-\xi)}
\tilde{a}_i(t)\,dt.
\end{align*}
Consequently,
\begin{align*}
\begin{aligned}
|V_i'(\xi)|
\leq
\sup_{t\geq\xi}|\tilde{a}_i(t)|
\int_{\xi}^{+\infty}
e^{-\beta_i(t-\xi)}\,dt=
\frac{1}{\beta_i}
\sup_{t\geq\xi}|\tilde{a}_i(t)|.
\end{aligned}
\end{align*}
Since $\tilde{a}_i(t)\to0$, it follows that
\begin{align*}
V_i'(\xi)\longrightarrow0
\qquad\text{as}\qquad \xi\to+\infty.
\end{align*}
Moreover,
\begin{align*}
V_i(\xi)\longrightarrow u_i^*>0.
\end{align*}
Therefore,
\begin{equation}\label{eq:q-plus-infinity}
\lim_{\xi\to+\infty}q_i(\xi)
=
0.
\end{equation}

Finally, by \eqref{eq:q-minus-infinity} and \eqref{eq:q-plus-infinity},
\begin{align*}
\lim_{\xi\to-\infty}q_i(\xi)
=
\frac{s}{D_i}
>
\frac{s}{2D_i}
>
0
=
\lim_{\xi\to+\infty}q_i(\xi).
\end{align*}
Define
\begin{align*}
\xi_i
:=
\sup
\left\{
\xi\in\mathbb R:
q_i(\xi)\geq\frac{s}{2D_i}
\right\}.
\end{align*}
The preceding limits imply that $\xi_i$ is finite.  By continuity,
\begin{align*}
q_i(\xi_i)=\frac{s}{2D_i}.
\end{align*}
Moreover, for every sufficiently small $h>0$,
\begin{align*}
q_i(\xi_i+h)
<
\frac{s}{2D_i}
=
q_i(\xi_i).
\end{align*}
Since $q_i$ is differentiable, we obtain
\begin{align*}
q_i'(\xi_i)\leq0.
\end{align*}

Evaluating \eqref{eq:consumer-riccati} at $\xi=\xi_i$, we find
\begin{align*}
\begin{aligned}
0\geq
q_i'(\xi_i)=
\frac{s}{D_i}
\frac{s}{2D_i}
-
\left(
\frac{s}{2D_i}
\right)^2
-
\frac{e_i}{D_i}
\bigl(
U_i(\xi_i)-V_i(\xi_i)
\bigr)=
\frac{s^2}{4D_i^2}
-
\frac{e_i}{D_i}
\bigl(
U_i(\xi_i)-V_i(\xi_i)
\bigr).
\end{aligned}
\end{align*}
Hence
\begin{align*}
\frac{s^2}{4D_i}
\leq
e_i
\bigl(
U_i(\xi_i)-V_i(\xi_i)
\bigr).
\end{align*}
In particular, the right-hand side is positive.  Notice that no global
assumption of the form $U_i\geq V_i$ is required.  At the crossing
point $\xi_i$, positivity of $U_i-V_i$ is forced by the Riccati
equation itself.

Using \eqref{eq:Ui-less-than-one} and $V_i(\xi_i)>0$, we obtain
\begin{align*}
U_i(\xi_i)-V_i(\xi_i)<1.
\end{align*}
Therefore,
\begin{align*}
\frac{s^2}{4D_i}<e_i,\qquad\text{and hence}\qquad s<2\sqrt{D_i e_i}.
\end{align*}
Since this inequality holds for every $i\in\N$, we conclude that
\begin{align*}
s
<
\min_{i\in\N}2\sqrt{D_i e_i}
=
s_{\mathrm C}^{\mathrm{nec}}.
\end{align*}
This completes the proof of Theorem~\ref{thm:upper-speed-nonexistence}.
\end{proof}

\section{An Explicit Wave Beyond the Constructive Threshold}
\label{sec:explicit-wave}

The existence result in Theorem~\ref{thm:existence} guarantees positive traveling waves for
\[
s_{\mathrm P}<s
<
s_{\mathrm C}^{\mathrm{ex}}.
\]
On the other hand, Theorem~\ref{thm:upper-speed-nonexistence} shows that every strictly positive
extinction-to-coexistence wave must satisfy
\[
s
<
s_{\mathrm C}^{\mathrm{nec}}.
\]
The purpose of this section is to show that the constructive threshold
$s_{\mathrm C}^{\mathrm{ex}}$ is not an intrinsic upper barrier. More
precisely, we exhibit a fixed two-species system admitting an explicit
traveling wave with speed strictly inside the intermediate regime
\begin{align*}
\left[
s_{\mathrm C}^{\mathrm{ex}},
s_{\mathrm C}^{\mathrm{nec}}
\right).
\end{align*}

Let us consider system \eqref{eq:pc-system} with $N=2$ and coefficients
\[
d_1=\frac{7}{120},
\qquad
d_2=\frac{1}{120},
\qquad
D_1=D_2=\frac16,
\]
\[
r_1=\frac{11}{10},
\qquad
r_2=\frac{13}{10},
\qquad
e_1=8,
\qquad
e_2=5,
\]
and
\[
K=
\begin{pmatrix}
\displaystyle \frac{28}{33}
&
\displaystyle \frac{25}{264}
\\[7pt]
\displaystyle \frac{58}{195}
&
\displaystyle \frac{5}{78}
\end{pmatrix}.
\]
We observe that system \eqref{eq:pc-system} admits a strictly positive traveling-wave
solution
\[
(u,v)(x,t)
=
(U,V)(\xi), \qquad \xi=x+st
\]
with speed
\[
s=s_{\mathrm{exact}}
=
\frac23,
\]
where
\[
U_1(\xi)
=
\frac14\bigl(1+\tanh \xi\bigr),
\qquad
V_1(\xi)
=
\frac18\bigl(1+\tanh \xi\bigr)^2,
\]
and
\[
U_2(\xi)
=
\frac25\bigl(1+\tanh \xi\bigr),
\qquad
V_2(\xi)
=
\frac15\bigl(1+\tanh \xi\bigr)^2.
\]
Let us prove Proposition~\ref{prop} as follows.
\begin{proof}[Proof of Proposition~\ref{prop}]
First, the row sums of $K$ are
\begin{align*}
\frac{28}{33}
+
\frac{25}{264}
=
\frac{83}{88}\qquad\text{and}\qquad\frac{58}{195}
+
\frac{5}{78}
=
\frac{47}{130}.
\end{align*}
Therefore,
\[
\|K\|_{\infty}
=
\max
\left\{
\frac{83}{88},
\frac{47}{130}
\right\}
=
\frac{83}{88}
<
1.
\]

In particular, the weak-interaction condition is satisfied. The corresponding
quantities
\[
\sigma_i
=
1-\sum_{j=1}^2\kappa_{ij}
\]
are given by
\begin{align*}
\sigma_1
=
1-\frac{83}{88}
=
\frac{5}{88}\qquad\text{and}\qquad\sigma_2
=
1-\frac{47}{130}
=
\frac{83}{130}.
\end{align*}

A straightforward calculation reveals that $(U_1,U_2,V_1,V_2,s_{\mathrm{exact}})$ is a positive solution of system~\eqref{eq:traveling-wave-system}. Moreover,
\begin{align*}
\lim_{\xi\to-\infty}(U,V)(\xi)
=
(0,0,0,0) \quad\text{and}\quad 
\lim_{\xi\to+\infty}U(\xi)
=
\lim_{\xi\to+\infty}V(\xi)
=
(I+K)^{-1}\begin{pmatrix}
1\\[3pt]
1
\end{pmatrix}.
\end{align*}
Finally, a direct computation reveals that
\begin{align*}
s_{\mathrm P}
=
2\sqrt{\frac{77}{1200}}<s_{\mathrm C}^{\mathrm{ex}}
=
2\sqrt{\frac{5}{66}}<s_{\mathrm{exact}}
=
\frac23<s_{\mathrm C}^{\mathrm{nec}}
=
2\sqrt{\frac56}.
\end{align*}
This completes the proof.
\end{proof}

\appendix
\section{Coexistence equilibrium} \label{CE}
\begin{lemma}[Unique positive coexistence equilibrium]
\label{lem:unique-positive-equilibrium}
Assume that $K\geq 0$ and
\begin{align*}
\Theta:=\|K\|_{\infty}<1.
\end{align*}
Then system~\eqref{eq:pc-system} admits a unique strictly positive
spatially homogeneous equilibrium
\begin{align*}
E^*=(u^*,v^*),
\quad
v^*=u^*,
\quad
u^*=(I+K)^{-1}\mathbf{1},\quad\text{with}\quad \mathbf{1}:=(1,\ldots,1)^{\mathsf T}.
\end{align*}
Moreover,
\begin{align*}
1-\Theta\leq u_i^*\leq 1,
\qquad i\in\N.
\end{align*}
\end{lemma}

\begin{proof}
A spatially homogeneous positive equilibrium $(u^*,v^*)$ satisfies
\begin{align*}
0=e_i v_i^*(u_i^*-v_i^*),
\qquad i\in\N.
\end{align*}
Hence $v^*=u^*$, and the plant equations become
\begin{align*}
u_i^*+\sum_{j\in\N}\kappa_{ij}u_j^*=1,
\qquad i\in\N,
\end{align*}
or equivalently,
\begin{align*}
(I+K)u^*=\mathbf{1}.
\end{align*}

Define
\begin{align*}
T(z):=\mathbf{1}-Kz,
\qquad z\in[0,1]^N.
\end{align*}
Since $K\geq0$ and $\Theta=\|K\|_\infty<1$, we have
\begin{align*}
1-\Theta\leq T_i(z)\leq1,
\qquad z\in[0,1]^N,\quad i\in\N.
\end{align*}
Thus
\begin{align*}
T([0,1]^N)\subset[1-\Theta,1]^N\subset[0,1]^N.
\end{align*}
Moreover,
\begin{align*}
\|T(z)-T(\widetilde z)\|_\infty
\leq
\Theta\|z-\widetilde z\|_\infty,
\qquad z,\widetilde z\in[0,1]^N.
\end{align*}
Therefore, by the Banach fixed-point theorem, $T$ has a unique fixed point
\begin{align*}
u^*\in[1-\Theta,1]^N.
\end{align*}
Setting $v^*=u^*$ gives a strictly positive spatially homogeneous
equilibrium of system~\eqref{eq:pc-system}.

Conversely, any strictly positive spatially homogeneous equilibrium
must satisfy $(I+K)u=\mathbf{1}$.  Since $K\geq0$, this identity implies $0<u_i\leq1$ for $i\in\N$, and hence $u\in[0,1]^N$.  By the uniqueness of the fixed point of $T$,
we obtain $u=u^*$ and $v=v^*$.  Finally, since
$\|K\|_\infty<1$, the matrix $I+K$ is invertible, and therefore
\begin{align*}
u^*=(I+K)^{-1}\mathbf{1},
\qquad
v^*=u^*.
\end{align*}
If, in addition, $\kappa_{ii}>0$ for every $i\in\N$, then $(1+\kappa_{ii})u_i^*\leq1$,
so that $u_i^*<1$ for every $i\in\N$.  Hence, $E^*=(u^*,v^*)\in(0,1)^{2N}$.
\end{proof}

\section{Sturm comparison lemma}\label{S thm}
We record the Sturm comparison lemma. A proof is included
for completeness.

\begin{lemma}\label{lem:sturm-comparison}
Let $R\in\mathbb{R}$, $\delta>0$, and let
\begin{align*}
    a\in C(( -\infty,R])
\end{align*}
satisfy
\begin{align*}
    a(\xi)\geq \delta,
    \qquad \xi\leq R.
\end{align*}
Then the equation
\begin{equation}
\label{eq:sturm-general}
    z''+a(\xi)z=0
\end{equation}
admits no strictly positive or strictly negative solution on $(-\infty,R]$.

More precisely, every nontrivial solution of \eqref{eq:sturm-general}
has a zero in each interval
\begin{align*}
    \left[\ell,\ell+\frac{\pi}{\sqrt{\delta}}\right]
    \subset(-\infty,R].
\end{align*}
\end{lemma}

\begin{proof}
Fix $\ell\in\mathbb{R}$ such that
\begin{align*}
    \ell+\frac{\pi}{\sqrt{\delta}}\leq R,
\end{align*}
and set
\begin{align*}
    L:=\frac{\pi}{\sqrt{\delta}},
    \qquad
    \phi(\xi):=\sin\bigl(\sqrt{\delta}(\xi-\ell)\bigr),
    \qquad
    \xi\in[\ell,\ell+L].
\end{align*}
Then
\begin{align*}
    \phi''+\delta\phi=0,
\end{align*}
and
\begin{align*}
    \phi(\ell)=\phi(\ell+L)=0,
    \qquad
    \phi(\xi)>0
    \quad\text{for }\xi\in(\ell,\ell+L).
\end{align*}
Moreover,
\begin{align*}
    \phi'(\ell)=\sqrt{\delta},
    \qquad
    \phi'(\ell+L)=-\sqrt{\delta}.
\end{align*}

Suppose, by contradiction, that a nontrivial solution $z$ of
\eqref{eq:sturm-general} has no zero in $[\ell,\ell+L]$.
Replacing $z$ by $-z$ if necessary, we may assume that
\begin{align*}
    z(\xi)>0,
    \qquad
    \xi\in[\ell,\ell+L].
\end{align*}
Multiplying \eqref{eq:sturm-general} by $\phi$, multiplying
\begin{align*}
    \phi''+\delta\phi=0
\end{align*}
by $z$, and subtracting the resulting identities, we obtain
\begin{align*}
    z''\phi-\phi''z+(a-\delta)z\phi=0.
\end{align*}
Integrating over $[\ell,\ell+L]$ yields
\begin{align*}
    0
    &=
    \left[z'\phi-z\phi'\right]_{\ell}^{\ell+L}
    +\int_{\ell}^{\ell+L}
      (a(\xi)-\delta)z(\xi)\phi(\xi)\,d\xi
    \\
    &=
    \sqrt{\delta}\bigl(z(\ell)+z(\ell+L)\bigr)
    +\int_{\ell}^{\ell+L}
      (a(\xi)-\delta)z(\xi)\phi(\xi)\,d\xi.
\end{align*}
The first term is strictly positive, whereas the second term is
nonnegative. This is a contradiction. Hence $z$ must vanish in
$[\ell,\ell+L]$.
\end{proof}

\section{Acknowledgement}
Chiun-Chuan Chen is supported by the National Science and Technology Council, Taiwan (Grant Number 114-2115-M-002-004-MY3) and the National Center for Theoretical Sciences, Taiwan (NCTS). Ting-Yang Hsiao is supported by the European Union ERC CONSOLIDATOR GRANT 2023 GUnDHam, Project Number: 101124921. Views and opinions expressed are however those of the authors only and do not necessarily reflect those of the European Union or the European Research Council. Neither the European Union nor the granting authority can be held responsible for them.

\begin{center}
\bibliographystyle{alpha}
\bibliography{Bibjournal.bib}
\end{center}

\end{document}